\documentclass[reqno, 10pt]{amsart}
\usepackage{amsmath, amsthm, amssymb, amstext}

\usepackage{mathrsfs}

\usepackage{hyperref,xcolor}
\hypersetup{
 pdfborder={0 0 0},
 colorlinks,
}
\usepackage[left=2.5cm,right=2.5cm,top=3cm,bottom=3cm]{geometry}

\usepackage[colorinlistoftodos]{todonotes}
\usepackage{enumitem}
\makeatletter \@addtoreset{equation}{section} \makeatother

\newtheorem{theorem}{Theorem}[section]

\newtheorem{proposition}[theorem]{Proposition}
\newtheorem{lemma}[theorem]{Lemma}
\newtheorem{remark}[theorem]{Remark}

\theoremstyle{definition}

\newtheoremstyle{case}{}{}{}{}{}{\em:}{ }{}
\theoremstyle{case}
\newtheorem{case}{\it Case}
\newtheorem{subcase}{\it Subcase}
\numberwithin{subcase}{case}

\allowdisplaybreaks

\DeclareMathOperator*{\divergenz}{div}

\newcommand{\RN}{\mathbb{R}^N}
\newcommand{\RR}{\mathbb R}
\newcommand{\NN}{\mathbb N}

\newcommand{\R}{\mathbb{R}}
\newcommand{\cH}{\mathcal{H}}
\newcommand{\nap}{\Vert \nabla u\Vert_p}
\newcommand{\naq}{\Vert \nabla u\Vert_{q,a}}

\newcommand{\rof}{\varrho_\cH(u)}
\newcommand{\roc}{\widehat{\varrho}_\cH(u)}

\newcommand{\eps}{\varepsilon}

\newcommand{\Om}{\Omega}
\newcommand{\rand}{\partial\Omega}

\newcommand{\into}{\int_{\Omega}}
\newcommand{\intor}{\int_{\partial\Omega}}

\newcommand{\Linf}{L^{\infty}(\Omega)}
\newcommand{\close}{\overline{\Omega}}

\newcommand{\soh}{W^{1,\mathcal{H}}(\Om)}
\newcommand{\normasoh}[1]{\Vert #1\Vert_{1,\mathcal{H}}}

\renewcommand{\l}{\left}
\renewcommand{\r}{\right}
\newcommand{\Hi}{\mathcal H}

\numberwithin{theorem}{section}
\numberwithin{equation}{section}
\definecolor{champagne}{rgb}{0.97, 0.91, 0.81}
\definecolor{aoo}{rgb}{0.0, 0.5, 0.0}

\begin{document}


\title[Multiple solutions to Kirchhoff type double phase problems]
{Multiple solutions for nonlinear boundary value problems of Kirchhoff type on a double phase setting}

\author[Alessio Fiscella]{Alessio Fiscella}
\address[Alessio Fiscella]{Dipartimento di Matematica e Applicazioni, Universit\`a degli Studi di Milano-Bicocca, Via Cozzi 55, Milano, CAP 20125, Italy}
\email{alessio.fiscella@unimib.it}

\author[Greta Marino]{Greta Marino}
\address[Greta Marino]{Technische Universit\"{a}t Chemnitz, Fakult\"{a}t f\"{u}r Mathematik, Reichenhainer Stra\ss e 41, 09126 Chemnitz, Germany}
\email{greta.marino@mathematik.tu-chemnitz.de}

\author[Andrea Pinamonti]{Andrea Pinamonti}
\address[Andrea Pinamonti]{Dipartimento di Matematica,	Universit\`a degli Studi di Trento,
	Via Sommarive 14, 38123, Povo (Trento), Italy}
\email{andrea.pinamonti@unitn.it}

\author[Simone Verzellesi]{Simone Verzellesi}
\address[Simone Verzellesi]{Dipartimento di Matematica,	Universit\`a degli Studi di Trento,
	Via Sommarive 14, 38123, Povo (Trento), Italy}
\email{simone.verzellesi@unitn.it}


\date{\today}
\maketitle


\begin{abstract} 
This paper deals with some classes of Kirchhoff type problems on a double phase setting and with nonlinear boundary conditions. Under general assumptions, we provide multiplicity results for such problems in the case when  the perturbations exhibit a suitable behavior in the origin and at infinity, or when they do not necessarily satisfy the Ambrosetti-Rabinowitz condition. To this aim, we combine variational methods, truncation arguments and topological tools.

\end{abstract}

\maketitle

\section{Introduction}\label{sec:introduction}

In this paper, we study multiplicity results for some classes of double phase problems of Kirchhoff type with nonlinear boundary conditions. More precisely, we consider the problems
\begin{equation}\label{P}
    \left\{
\begin{aligned}
&-M\left[\displaystyle\int_\Omega\left(\frac{|\nabla u|^p}{p}+a(x)\frac{|\nabla u|^q}{q}\right)dx\right]
\mbox{div}\left(|\nabla u|^{p-2}\nabla u+a(x)|\nabla u|^{q-2}\nabla u\right)= h_1(x,u) \quad && \mbox{in } \Omega,\\[1.1ex]
&M\left[\displaystyle\int_\Omega\left(\frac{|\nabla u|^p}{p}+a(x)\frac{|\nabla u|^q}{q}\right)dx\right]\left(|\nabla u|^{p-2}\nabla u+a(x)|\nabla u|^{q-2}\nabla u\right)\cdot \nu = h_2(x,u) && \mbox{on } \partial\Omega,
\end{aligned}
\right.
\end{equation}
and 
\begin{equation}\label{P2}
    \left\{
\begin{aligned}
&-M\l(\into |\nabla u|^p dx\r) \Delta_p u- M\l(\into a(x) |\nabla u|^q dx \r) \divergenz \l(a(x) |\nabla u|^{q-2} \nabla u\r)= h_1(x, u) \quad && \text{in } \Om, \\[1ex]
&\left[M\l(\into |\nabla u|^p dx \r)|\nabla u|^{p-2}\nabla u+M\l(\into a(x) |\nabla u|^q dx \r)a(x)|\nabla u|^{q-2}\nabla u\right]\cdot \nu = h_2(x,u) && \mbox{on } \partial\Omega,
\end{aligned}
\right.
\end{equation}
where \textit{along the paper, and without further mentioning,} $\Om \subset \RR^N$, $N>1$, is a bounded domain with $C^{1, \alpha}$-  boundary $\rand$, $\alpha \in (0, 1]$, $\nu(x)$ is the outer unit normal of $\Om$ at the point $x \in \rand$, $1<p<q<N$, and
	\begin{equation}
	\label{condizione1}
	 q<p^*, \quad a\colon \close \to [0, \infty) \text{ is in }L^\infty(\Omega),
	\end{equation}
with $p^*=Np/(N-p)$ being the critical Sobolev exponent, see \eqref{critical-exp} for its definition.
We assume that $M\colon [0, \infty) \to [0, \infty)$ is a \textit{continuous} function satisfying the following assumptions:

\begin{enumerate}
\item[$(M_1)$]
\textit{there exists} $\theta\geq1$ \textit{such that} $t M(t)\le \theta\mathscr M(t)$ \textit{for all} $t\in[0,\infty)$,
\textit{where} $\mathscr M(t)=\displaystyle\int_0^t M(\tau)\;d\tau$;
\item[$(M_2)$]
\textit{for all} $\tau>0$ \textit{there exists} $\kappa=\kappa(\tau)>0$ \textit{such that} $M(t)\geq \kappa$
\textit{for all} $t\geq \tau$.
\end{enumerate}

\vspace{0.2cm}

Moreover, $h_1\colon \Om \times \RR \to \RR$ and $h_2\colon \rand \times \RR \to \RR$ are \textit{Carath\'eodory} functions whose properties will be specified in the sequel, case by case. A classical model for $M$ verifying $(M_1)$-$(M_2)$ due to Kirchhoff is given by $M(t)=a+bt^{\theta-1}$, where $a$, $b\geq0$ and $a+b>0$.

Problems \eqref{P} and \eqref{P2} are said of double phase type because of the presence of two different elliptic growths $p$ and $q$. The study of double phase problems and related functionals originates from the seminal paper by Zhikov \cite{Z1}, where he introduced the functional
\begin{equation}\label{funzionale}
u\mapsto\int_\Omega\left(|\nabla u|^p+a(x)|\nabla u|^q\right)dx\quad\mbox{with }1<p<q,\quad a(\cdot)\geq0,
\end{equation}
in order to provide models for strongly anisotropic materials. Indeed, the weight coefficient $a(\cdot)$ dictates the geometry of composites made of two different materials with distinct power hardening exponents $p$ and $q$. From the mathematical point of view, \eqref{funzionale} is a prototype of a functional whose integrands change their ellipticity according to the points where $a(\cdot)$ vanishes or not. In this direction, Zhikov found other mathematical applications for \eqref{funzionale} in the study of duality theory and of the Lavrentiev gap phenomenon, as shown in \cite{Z2,Z3,ZK}. Furthermore, \eqref{funzionale} falls into the class of functionals with non-standard growth conditions, according to Marcellini's definition given in \cite{M1,M2}. Following this line of research, Mingione et al. provide different regularity results for minimizers of \eqref{funzionale}, see \cite{BCM,BCM2,CM1,CM2}. In \cite{CS}, Colasuonno and Squassina analyze the eigenvalue problem with Dirichlet boundary condition of the double phase operator, explicitly appearing in \eqref{P}, that is,
\begin{equation}\label{operatore}
u\mapsto\mbox{div}\left(|\nabla u|^{p-2}\nabla u+a(x)|\nabla u|^{q-2}\nabla u\right),
\end{equation}
whose energy functional is given by \eqref{funzionale}. 

Starting from \cite{CS}, several authors studied existence and multiplicity results for nonlinear problems driven by \eqref{operatore}, such as in \cite{CGHW,EMW,FFW,FW,F,FP,GW1,GLL,LD,PS} with the help of different variational techniques. In particular, in \cite{FP} Fiscella and Pinamonti provide existence and multiplicity results for Kirchhoff double phase problems but with Dirichlet boundary condition. That is, in \cite{FP} we have
\begin{equation}\label{P3}
\left\{\begin{array}{ll}
-M\left[\displaystyle\int_\Omega\left(\frac{|\nabla u|^p}{p}+a(x)\frac{|\nabla u|^q}{q}\right)dx\right]
\mbox{div}\left(|\nabla u|^{p-2}\nabla u+a(x)|\nabla u|^{q-2}\nabla u\right)=f(x,u) & \mbox{in } \Omega,\\
u=0 & \mbox{in } \partial\Omega,
\end{array}
\right.
\end{equation}
and 
$$
\left\{\begin{array}{ll}
-M\left(\displaystyle\int_\Omega|\nabla u|^p dx\right)\Delta_p u-M\left(\displaystyle\int_\Omega a(x)|\nabla u|^q dx\right)
\mbox{div}\left(a(x)|\nabla u|^{q-2}\nabla u\right)=f(x,u) & \mbox{in } \Omega,\\
u=0 & \mbox{in } \partial\Omega,
\end{array}
\right.
$$
where the nonlinear subcritical term $f$ satisfies the classical Ambrosetti-Rabinowitz (AR) condition, crucial to prove the boundedness of the so-called Palais-Smale sequences.

The aim of the present paper is to provide different multiplicity results for some classes of \eqref{P} and \eqref{P2}, that will heavily depend on the properties of the involved nonlinearities. More specifically, in the first part of the work, we will study existence of at least two constant sign solutions (more precisely, one nonnegative and one nonpositive), under the assumption that the nonlinearities satisfy some strong conditions such as the superlinearity at $\pm \infty$. These results are inspired by the truncation arguments used in \cite{EMW}, where the authors studied existence of constant sign solutions in a double phase setting but without the Kirchhoff coefficient, namely, with $M \equiv 1$. However, the presence of the nonlocal Kirchhoff coefficient in \eqref{P} and \eqref{P2} makes the comparison analysis more intriguing than \cite{EMW}. For this, our multiplicity result will depend on the value $M(0)$, beside the first eigenvalue of the Robin eigenvalue problem for the $p$-Laplacian, as required in \cite{EMW}. Other results for existence of solutions with sign information in a double phase setting can be found, for example, in \cite{GW2019, PRR, GP}.

In the second part of the article we will study multiplicity of solutions in the case when the nonlinearities on the right hand side do not necessarily satisfy the (AR) condition. In this direction, we can mention \cite{GW1,GLL} where they consider problems \eqref{P} and \eqref{P3} with $M\equiv1$, namely without the Kirchhoff coefficient. In \cite{GLL}, Ge et al. exploit the classical mountain pass theorem and the Krasnoselskii's genus theory in order to prove existence and multiplicity results for \eqref{P3}, replacing the (AR) condition with
\begin{equation}\label{ipotesi}
\mathcal F(x,t)\leq\mathcal F(x,s)+C_0,\quad\mbox{where }\mathcal F(x,t)=tf(x,t)-qF(x,t)\mbox{ and }C_0>0
\end{equation}
for a.e. $x\in\Omega$, $0<t<s$ or $s<t<0$, where $F$ is the primitive of $f$ with respect the second variable. On the other hand, in \cite{GW1} Gasi\'nski and Winkert prove the existence of two solutions of \eqref{P} with $h_1(x,t)=f(x,t)-|t|^{p-2}t-a(x)|t|^{q-2}t$, with  the following quasi-monotonic assumption for $f$
\begin{equation}\label{ipotesi2}
t\mapsto tf(x,t)-qF(x,t)\mbox{ is nondecreasing in }\mathbb R^+\mbox{ and nonincreasing in }\mathbb R^-,
\end{equation}
for a.e. $x\in\Omega$. Because of the presence of a nonlocal Kirchhoff coefficient $M$ in \eqref{P}, similar hypothesis to \eqref{ipotesi} and \eqref{ipotesi2} can not work, even if we assume a monotonic assumption for $t\mapsto \theta\mathscr M(t)-M(t)t$. This is a consequence of the fact that $M$ in \eqref{P} does not depend on the norm of $W^{1,\mathcal H}(\Omega)$, that is the functional space  where we look for the solutions to \eqref{P} and \eqref{P2}. Indeed, as shown in detail in Section \ref{sec2}, $W^{1,\mathcal H}(\Omega)$ is a suitable Musielak-Orlicz Sobolev space endowed with a norm of Luxemburg type. For this reason, we need different conditions for the nonlinearities in \eqref{P} and \eqref{P2} than in \cite{GW1,GLL}.

The paper is organized as follows. In Section \ref{sec2}, we recall the main properties of Musielak-Orlicz Sobolev spaces $W^{1,\mathcal H}(\Omega)$
 and we state some technical lemmas concerning the Kirchhoff coefficient $M$. In Section \ref{sec3}, for each problem we prove the existence of two constant sign solutions, more precisely one nonnegative and one nonpositive, combining the ideas of \cite{EMW, FP}. In order to show the results, we require some strong conditions on the nonlinearities, like the superlinearity at $\pm \infty$. We also point out that the proofs depend on the first eigenvalue of the Robin eigenvalue problem for the $p$-Laplacian. 
In the last part of the paper, namely Section \ref{sec4}, we show the existence of infinitely many solutions to a class of \eqref{P} and \eqref{P2}. In particular, we introduce suitable assumptions for the nonlinearities in order to avoid the (AR) condition.

\section{Preliminaries}\label{sec2}

In this section we introduce the basic notation of our paper, the functional space where we find the solutions to \eqref{P}, \eqref{P2} and we present some technical lemmas for $M$ that will be used in the sequel. 

For all $1 \le r< \infty$ we denote by $L^r(\Om)$ and $L^r(\Om; \RN)$ the usual Lebesgue spaces equipped with the norm $\|\cdot\|_r$. The corresponding Sobolev space is denoted by $W^{1,r}(\Om)$, endowed with the norm $\|\cdot\|_{1,r}$. 
On the boundary $\rand$ of $\Om$ we consider the $(N-1)$-dimensional Hausdorff measure $\sigma$ and denote by $L^r(\rand)$ the boundary Lebesgue space with corresponding norm $\|\cdot\|_{r, \rand}$.
As matter of notations, for all $t\in \RR $ we set $t^{\pm}:= \max\{\pm t, 0\}$ and similarly we denote by $u^{\pm}(\cdot):= u(\cdot)^{\pm}$ the positive and negative parts of a function $u$.

The function $\mathcal H\colon \Omega\times[0,\infty)\to[0,\infty)$ defined as
	\[
	\mathcal H(x,t):=t^p+a(x)t^q \quad\mbox{for a.e. }x\in\Omega\mbox{ and for all }t\in[0,\infty),
	\]
with $1<p<q$ and $0\leq a \in L^1(\Omega)$, is a generalized N-function (N stands for {\em nice}), according to the definition in \cite{D,M}, and satisfies the so-called $(\Delta_2)$ condition, that is,
	\[
	\mathcal H(x,2t)\leq 2^q\mathcal H(x,t) \quad\mbox{for a.e. }x\in\Omega\mbox{ and for all }t\in[0,\infty).
	\]
Then, we introduce the $\mathcal H$-modular function $\varrho_{\mathcal H}$ given by
	\begin{equation}\label{rhoh}
	\varrho_{\mathcal H}(u):=\int_\Omega\mathcal H(x,|u|) \;dx=\int_\Omega\left(|u|^p+a(x)|u|^q\right) dx.
	\end{equation}
Therefore, by \cite{M} we can define the Musielak-Orlicz space $L^{\mathcal H}(\Omega)$ as
	\[
	L^{\mathcal H}(\Omega):=\left\{u:\Omega\to\overline{\mathbb R}\,\bigl| \, u\mbox{ is measurable and } \varrho_{\mathcal H}(u)<\infty\right\},
	\]
endowed with the Luxemburg norm
	\[
	\|u\|_{\mathcal H}:=\inf\left\{\lambda>0:\,\,\varrho_{\mathcal H}\left(\frac{u}{\lambda}\right)\leq1\right\}.
	\]
From \cite{CS,D}, the space $L^{\mathcal H}(\Omega)$ is a separable, uniformly convex, Banach space.
On the other hand, from \cite[Proposition 2.1]{LD} we have the following relation between the norm $\|\cdot\|_{\mathcal H}$ and the $\mathcal H$-modular function.

\begin{proposition}\label{P2.1}
Assume that $u\in L^{\mathcal H}(\Omega)$, $(u_j)_j\subset
L^{\mathcal H}(\Omega)$, and $c>0$. Then,
	\begin{itemize}
	\item[$(i)$] If $ u\neq0$, then $ \|u\|_{\mathcal H}=c$ if and only if $\displaystyle \varrho_{\mathcal H}\left(\frac{u}{c}\right)=1$;
	\vspace{0.1cm}
	\item[$(ii)$] $\|u\|_{\mathcal H}<1$ $(resp.=1,\,>1)$ if and only if  $ \varrho_{\mathcal H}(u)<1$ $(resp.=1,\,>1)$;
	\vspace{0.1cm}
	\item[$(iii)$] If $\|u\|_{\mathcal H}<1$, then $ \|u\|_{\mathcal H}^q\leq\varrho_{\mathcal H}(u)\leq\|u\|_{\mathcal H}^p$;
	\vspace{0.1cm}
	\item[$(iv)$] If $\|u\|_{\mathcal H}>1$, then $\|u\|_{\mathcal H}^p\leq\varrho_{\mathcal H}(u)\leq\|u\|_{\mathcal H}^q$;
	\vspace{0.1cm}
	\item[$(v)$] $\lim\limits_{j\rightarrow\infty}\|u_j\|_{\mathcal H}=0\,(\text{resp. } \infty) $ if and only if  $\lim\limits_{j\rightarrow\infty}\varrho_{\mathcal H}(u_j)=0\,(\text{resp. } \infty)$.
	\end{itemize}
\end{proposition}

The related Musielak-Orlicz Sobolev space $W^{1,\mathcal H}(\Omega)$ is defined by
$$W^{1,\mathcal H}(\Omega):=\left\{u\in L^{\mathcal H}(\Omega):\,\,|\nabla u|\in L^{\mathcal H}(\Omega)\right\},
$$
endowed with the norm
\begin{equation}\label{norma}
\|u\|_{1,\mathcal H}:=\|u\|_{\mathcal H}+\||\nabla u|\|_{\mathcal H}.
\end{equation}
Along the paper, we write $\|\nabla u\|_{\mathcal H}:=\||\nabla u|\|_{\mathcal H}$ and $\varrho_{\mathcal H}(\nabla u):=\varrho_{\mathcal H}(|\nabla u|)$ in order to simplify the notation.

We also define the complete
$\mathcal H$-modular function $\widehat{\varrho}_{\mathcal H}$ as
\begin{equation}\label{rhotilde}
	\widehat{\varrho}_{\mathcal H}(u):=\varrho_{\mathcal H}(\nabla u)+\varrho_{\mathcal H}(u)
	=\int_\Omega\left(|\nabla u|^p+a(x)|\nabla u|^q+|u|^p+a(x)|u|^q\right) dx.
\end{equation}
Hence, following the proof of \cite[Proposition 2.1]{LD} we can relate $\widehat \varrho_{\cH}$ and the norm $\|\cdot\|_{1, \cH}$ as follows. 

\begin{proposition}\label{P2.2}
Assume that $u\in W^{1,\mathcal H}(\Omega)$, $(u_j)_j\subset
W^{1,\mathcal H}(\Omega)$, and $c>0$. Then,
	\begin{itemize}
	\item[$(i)$] If $ u\neq0$, then $ \|u\|_{1,\mathcal H}=c$ if and only if $\displaystyle \widehat{\varrho}_{\mathcal H}\left(\frac{u}{c}\right)=1$;
	\vspace{0.1cm}
	\item[$(ii)$] $\|u\|_{1,\mathcal H}<1$ $(resp.=1,\,>1)$ if and only if  $ \widehat{\varrho}_{\mathcal H}(u)<1$ $(resp.=1,\,>1)$;
	\vspace{0.1cm}
	\item[$(iii)$] If $\|u\|_{1,\mathcal H}<1$, then $ \|u\|_{1,\mathcal H}^q\leq\widehat{\varrho}_{\mathcal H}(u)\leq\|u\|_{1,\mathcal H}^p$;
	\vspace{0.1cm}
	\item[$(iv)$] If $\|u\|_{1,\mathcal H}>1$, then $\|u\|_{1,\mathcal H}^p\leq\widehat{\varrho}_{\mathcal H}(u)\leq\|u\|_{1,\mathcal H}^q$;
	\vspace{0.1cm}
	\item[$(v)$] $\lim\limits_{j\rightarrow\infty}\|u_j\|_{1,\mathcal H}=0\,(\text{resp. } \infty) $ if and only if  $\lim\limits_{j\rightarrow\infty}\widehat{\varrho}_{\mathcal H}(u_j)=0\,(\text{resp. } \infty)$.
	\end{itemize}
\end{proposition}

For all $1< p< N$ let $p^*$ and $p_*$ denote the critical Sobolev exponents of $p$, defined as
	\begin{equation}
	\label{critical-exp}
	p^*=
	\displaystyle \frac{Np}{N-p}
	\quad \text{as well as} \quad
	p_*=
	\displaystyle \frac{(N-1)p}{N-p}.
	\end{equation}
Furthermore, we define the weighted space
	\[
	L^q_a(\Omega):=\left\{u:\Omega\to\mathbb R\,\bigl| \, u\mbox{ is measurable and } \int_\Omega a(x)|u|^q \;dx<\infty\right\},
	\]
equipped with the seminorm
	\[
	\|u\|_{q,a}:=\left(\int_\Omega a(x)|u|^q \;dx\right)^{1/q}.
	\]
In a similar way we can define $L^q_a(\Om; \RN)$ and the associated seminorm.
 Thus, by \cite[Proposition 2.1]{EMW} we have the following embeddings for $L^\cH(\Om)$ and $\soh$.

\begin{proposition}\label{P2.3}
Let \eqref{condizione1} be satisfied. Then, the following embeddings hold:
\begin{itemize}
    \item [$(i)$] $L^\cH(\Om)\hookrightarrow L^r(\Om)$ and $\soh\hookrightarrow W^{1,r}(\Om)$ are continuous for all $r\in[1,p]$;
    \item [$(ii)$] $\soh\hookrightarrow L^{\nu_1}(\Om)$ is compact for all $\nu_1\in[1,p^*)$;
    \item [$(iii)$] $\soh\hookrightarrow L^{\nu_2}(\partial\Om)$ is compact for all $\nu_2\in[1,p_*)$;
    \item [$(iv)$] $L^q(\Om)\hookrightarrow L^\cH(\Om)\hookrightarrow L_a^q(\Om)$ are continuous;
    \item[$(v)$] $\soh\hookrightarrow L^\cH(\Om)$ is compact.
    \end{itemize}
\end{proposition}

Let us define the operator $L:W^{1,\mathcal H}(\Omega)\to\left(W^{1,\mathcal H}(\Omega)\right)^*$ such that
	\begin{equation}\label{essepiu}
	   	\langle L(u),v\rangle_{\mathcal H}:=\int_\Omega\left(|\nabla u|^{p-2}+a(x)|\nabla u|^{q-2}\right)\nabla u\cdot\nabla v \;dx, 
	\end{equation}
for all $u$, $v\in W^{1,\mathcal H}(\Omega)$.
Here, $\left(W^{1,\mathcal H}(\Omega)\right)^*$ denotes the dual space of $W^{1,\mathcal H}(\Omega)$ and $\langle \cdot \,,\cdot\rangle_{\mathcal H}$ is the related dual pairing. Then we have the following  result, see \cite[Proposition 3.1-(ii)]{LD}.

\begin{proposition}\label{P2.4}
The mapping $L\colon W^{1,\mathcal H}(\Omega)\to\left(W^{1,\mathcal H}(\Omega)\right)^*$ is  of $(S_+)$ type, that is, if $u_j\rightharpoonup u$ in $W^{1,\mathcal H}(\Omega)$ and $\limsup\limits_{j\to\infty}\langle L(u_j)-L(u),u_j-u\rangle\leq0$, then $u_j\to u$ in $W^{1,\mathcal H}(\Omega)$.
\end{proposition}

Now, we introduce some technical results related to the Kirchhoff coefficient $M$, under the assumptions that $(M_1)$-$(M_2)$ hold true. To this aim, we observe that integrating $(M_1)$ on $(1,t)$, when $t\geq 1$, gives
\begin{align}\label{mepsilon}
\mathscr M(t)\leq \mathscr M(1) t^\theta\qquad\mbox{for all }t\geq1.
\end{align}
Moreover, we set
	\begin{equation}
	\label{phi-H}
	\phi_{\mathcal H}(u):= \into \l( \frac{|u|^p}{p}+ a(x) \frac{|u|^q}{q} \r) dx.
	\end{equation}
Thus, we have the following lemmas.

\begin{lemma}
Let $u\in\soh$ be such that $\normasoh{u}>1$. Then, there exist ${A_1}$, ${A_2}>0$ such that
\begin{equation}\label{stimaA1}
    \varrho_{\cH}(\nabla u) M[\phi_{\cH}(\nabla u)]+ \varrho_{\cH}(u) \ge {A_1}\|u\|_{1, \cH}^p,
\end{equation}
and
\begin{equation}\label{stimaA2}
   M(\|\nabla u\|_p^p)\|\nabla u\|_p^p+M(\|\nabla u\|_{q,a}^q)\|\nabla u\|_{q,a}^q+\varrho_{\mathcal H}(u)
\geq {A_2}\|u\|_{1,\mathcal H}^p.
\end{equation}
\end{lemma}
\begin{proof}
We first prove \eqref{stimaA1}.
Since $\|u\|_{1, \cH} > 1$, by Proposition \ref{P2.2}-(ii) and (iv) we have that
	\[
	\widehat \varrho_{\cH}(u)\geq\|u\|_{1, \cH}^p>1.
	\]
This in particular implies that
	\begin{equation}
	\label{coerc3}
	\|u\|_{1,\mathcal H}^p\leq
	\begin{cases}
	\varrho_{\mathcal H}(\nabla u)+\varrho_{\mathcal H}(u)
	&\mbox{if $\varrho_{\mathcal H}(\nabla u)\geq1/2$},\\
	\varrho_{\mathcal H}(\nabla u)+\varrho_{\mathcal H}(u)\leq2\varrho_{\mathcal H}( u)
	&\mbox{if $\varrho_{\mathcal H}(\nabla u)<1/2$ and $\varrho_{\mathcal H}(u)\geq1/2$}.
	\end{cases}
	\end{equation}
First suppose that $\varrho_{\cH}(\nabla u) \ge 1/2$. Then it follows that
	\[
	\phi_{\cH}(\nabla u) \ge \frac{1}{q} \varrho_{\cH}(\nabla u) \ge \frac{1}{2q}.
	\]
Therefore we can use hypothesis ($M_2$) with $\displaystyle \tau= \frac{1}{2q}$ to find $\displaystyle \kappa= \kappa \l(\frac{1}{2q}\r)>0$ such that 
	\begin{equation}
	\label{coerc4}
	M(t) \ge \kappa \quad \text{for all } t \ge \frac{1}{2q}.
	\end{equation}
Then equation \eqref{coerc4} and the first line of equation \eqref{coerc3} give
	\begin{equation}
	\label{coerc5}
	\begin{split}
	\varrho_{\cH}(\nabla u) M[\phi_{\cH}(\nabla u)]+ \varrho_{\cH}(u) &\ge \kappa \varrho_{\cH}(\nabla u)+ \varrho_{\cH}(u) \\
	& \ge \min\{1, \kappa\} \widehat \varrho_{\cH}(u)  \\
	& \ge \min\{1, \kappa\} \|u\|_{1, \cH}^p.
	\end{split}
	\end{equation}
Suppose now that $\varrho_{\cH}(\nabla u)< 1/2$ and $\varrho_{\cH}(u) \ge 1/2$. Then the nonnegativity of $M$ and the second line of \eqref{coerc3} give
	\begin{equation}
	\label{coerc7}
	\varrho_{\cH}(\nabla u) M[\phi_{\cH}(\nabla u)]+ \varrho_{\cH}(u) \ge \frac{1}{2} \|u\|_{1, \cH}^p.
	\end{equation}
Combining \eqref{coerc5} and \eqref{coerc7} we get \eqref{stimaA1}.	
	
Now, we prove \eqref{stimaA2}. When $\varrho_{\mathcal H}(\nabla u)<1/2$, by \eqref{coerc3} we obtain
\begin{equation}\label{bassonuovo1}
M(\|\nabla u\|_p^p)\|\nabla u\|_p^p+M(\|\nabla u\|_{q,a}^q)\|\nabla u\|_{q,a}^q+\varrho_{\mathcal H}(u)
\geq\frac{1}{2}\|u\|_{1,\mathcal H}^p.
\end{equation}
On the other hand, if $\varrho_{\mathcal H}(\nabla u)\geq1/2$ we distinguish three situations: either $\|\nabla u\|^p_p\geq1/4$ and $\|\nabla u\|^q_{q,a}\geq1/4$; or $\|\nabla u\|^p_p\geq1/4$ and $\|\nabla u\|^q_{q,a}<1/4$, which yields $2\|\nabla u\|_p^p\geq\|\nabla u\|_p^p+\|\nabla u\|_{q,a}^q$; or $\|\nabla u\|^p_p<1/4$ and $\|\nabla u\|^q_{q,a}\geq1/4$, which yields $2\|\nabla u\|_{q,a}^q\geq\|\nabla u\|_p^p+\|\nabla u\|_{q,a}^q$.
Thus, by $(M_2)$ and Proposition \ref{P2.1}, we get
\begin{equation}\label{bassonuovo2}
M(\|\nabla u\|_p^p)\|\nabla u\|_p^p+M(\|\nabla u\|_{q,a}^q)\|\nabla u\|_{q,a}^q\geq
\begin{cases}
\kappa\varrho_{\mathcal H}(\nabla u),
&\mbox{if $\|\nabla u\|^p_p\geq1/4$ and $\|\nabla u\|^q_{q,a}\geq1/4$},\\
\frac{\kappa}{2}\varrho_{\mathcal H}(\nabla u),
&\mbox{if $\|\nabla u\|^p_p\geq1/4$ and $\|\nabla u\|^q_{q,a}<1/4$},\\
\frac{\kappa}{2}\varrho_{\mathcal H}(\nabla u),
&\mbox{if $\|\nabla u\|^p_p<1/4$ and $\|\nabla u\|^q_{q,a}\geq1/4$},
\end{cases}
\end{equation}
with $\kappa=\kappa(1/4)>0$ given in $(M_2)$ with $\tau=1/4$.
Combining \eqref{bassonuovo1} and \eqref{bassonuovo2} gives \eqref{stimaA2}.
\end{proof}

\begin{lemma}
    Let $u\in\soh$ be such that $\normasoh{u}>1$. Then, there exist ${B_1}$, ${B_2}>0$ such that
\begin{equation}\label{stimaB1}
    \mathscr{M}[\phi_{\cH}(\nabla u)]+ \varrho_{\cH}(u) \leq {B_1}(1+\|u\|_{1, \cH}^q+\|u\|_{1, \cH}^{q\theta})
\end{equation}
and
\begin{equation}\label{stimaB2}
    \mathscr{M}(\nap^p)+\mathscr{M}(\naq^q)+ \varrho_{\cH}(u) \leq {B_2}(1+\|u\|_{1, \cH}^q+\|u\|_{1, \cH}^{q\theta}).
\end{equation}
\end{lemma}
\begin{proof}
Let us fix $u\in\soh$ such that $\normasoh{u}>1$. We first prove \eqref{stimaB1}.

If $\varrho_{\mathcal H}(\nabla u)\geq q$, then $\phi_\cH(\nabla u)\geq\frac{1}{q}\varrho_\cH(\nabla u)\geq 1$, and so by \eqref{mepsilon} and Proposition \ref{P2.1}, we get
\begin{equation}\label{sopra1}
    \begin{split}
        \mathscr M[\phi_{\mathcal H}(\nabla u)]\leq\displaystyle\frac{\mathscr M(1)}{p^\theta}[\varrho_{\mathcal H}(\nabla u)]^{\theta}\leq
\displaystyle\frac{\mathscr M(1)}{p^\theta}\|\nabla u\|_{\mathcal H}^{q\theta}.
    \end{split}
\end{equation}
If conversely $\varrho_{\mathcal H}(\nabla u)\leq q$, then $\phi_\cH(\nabla u)\leq\varrho_{\mathcal H}(\nabla u)\leq q$. Therefore 
\begin{equation}\label{sopra2}
    \begin{split}
        \mathscr M[\phi_{\mathcal H}(\nabla u)]\leq\displaystyle\mathcal{M},
    \end{split}
\end{equation}
where $\displaystyle\mathcal M=\max_{t\in[0,q]}\mathscr M(t)\in(0,\infty)$ by $(M_2)$ and the continuity of $M$.
Therefore, combining \eqref{sopra1}, \eqref{sopra2} with Proposition \ref{P2.1} we obtain 
\begin{equation*}
    \begin{split}
        \mathscr M[\phi_{\mathcal H}(\nabla u)]+\varrho_{\cH}(u)&\leq
\mathcal{M}+\displaystyle\frac{\mathscr M(1)}{p^\theta}\|\nabla u\|_{\mathcal H}^{q\theta}+\varrho_{\cH}(u)\leq\mathcal{M}+\displaystyle\mathscr M(1)\|u\|_{1,\mathcal H}^{q\theta}+\|u\|_{1,\mathcal H}^{q}.
    \end{split}
\end{equation*}

In order to prove \eqref{stimaB2} we observe that, since $\normasoh{u}>1$, from Proposition \ref{P2.2} we have that $\widehat{\varrho}_\cH(u)\leq\normasoh{u}^q$.
We distinguish four cases. If $\Vert\nabla u\Vert_p\geq 1$ and $\Vert\nabla u\Vert_{q,a}\geq 1$, then 
\begin{equation}\label{sotto1}
    \begin{split}
        \mathscr{M}(\nap^p)+\mathscr{M}(\naq^q)+\rof&\leq\mathscr{M}(1)\nap^{p\theta}+\mathscr{M}(1)\naq^{q\theta}+\rof\\
        &\leq 2\mathscr{M}(1)\roc^\theta+\roc\\
        &\leq \max\{1,2\mathscr{M}(1)\}(\normasoh{u}^{q\theta}+\normasoh{u}^q).
    \end{split}
\end{equation}
If $\nap<1$ and $\naq\geq 1$ it holds that
\begin{equation}\label{sotto2}
    \begin{split}
        \mathscr{M}(\nap^p)+\mathscr{M}(\naq^q)+\rof
        &\leq \widetilde{\mathcal{M}}+\mathscr{M}(1)\naq^{q\theta}+\rof\\
        &\leq \widetilde{\mathcal{M}}+\mathscr{M}(1)\roc^\theta+\roc\\
        &\leq \max\{1,\widetilde{\mathcal{M}},\mathscr{M}(1)\}(1+\normasoh{u}^{q\theta}+\normasoh{u}^q),
    \end{split}
\end{equation}
with  $\displaystyle\widetilde{\mathcal M}=\max_{t\in[0,1]}\mathscr M(t)$.
If $\nap\geq 1$ and $\naq<1$ the estimate is the same as above. Finally, when $\nap<1$ and $\naq<1$, it holds that 
\begin{equation}\label{sotto3}
\mathscr{M}(\nap^p)+\mathscr{M}(\naq^q)+\rof\leq 2\widetilde{\mathcal{M}}+\roc\leq\max\{2\widetilde{\mathcal{M}},1\}(1+\normasoh{u}^q).
\end{equation}
Combining \eqref{sotto1}-\eqref{sotto3} we get \eqref{stimaB2}.
\end{proof}

\medskip

We now recall some basic properties on the spectrum of the negative $p$-Laplacian with Robin boundary condition. We refer to the paper of L$\hat{\text{e}}$ \cite{Le} for further details. The $p$-Laplacian eigenvalue problem with Robin boundary condition is given by
\begin{equation}\label{robin}
	\left\{
	\begin{aligned}
	-\Delta_p u&= \lambda |u|^{p-2} u \quad && \text{in } \Om, \\
	|\nabla u|^{p-2} \nabla u \cdot \nu &= -\beta |u|^{p-2} u && \text{on } \rand,
	\end{aligned}
	\right.
	\end{equation}
with $\beta>0$. 
It is well known that there exists a smallest eigenvalue  $\lambda_{1,p}>0$ of \eqref{robin} which is isolated, simple, and can be variationally characterized by
	\begin{equation}
	\label{var-robin}
	\lambda_{1,p}= \inf_{u \in W^{1,p}(\Om), u \ne 0} \frac{\displaystyle \into |\nabla u|^p dx+ \beta \intor |u|^p d\sigma}{\displaystyle \into |u|^p dx}.
	\end{equation}
Moreover, let $u_{1,p}$ be  the normalized (that is, $\|u_{1,p}\|_p= 1$) positive eigenfunction corresponding to $\lambda_{1,p}$. It is known that $u_{1,p} \in \text{int} \l(C^1(\close)_+\r)$, where
	\[
	\text{int} \l(C^1(\close)_+\r):= \l\{u \in C^1(\close): \, u(x)> 0 \text{ for all } x \in \close\r\}.
	\]

\medskip

We conclude this section with a result which will enable us to obtain the existence of  infinitely many solutions to \eqref{P} and \eqref{P2}.
To this aim, let $X$ be a Banach space. We recall that a functional $E\colon X \to\mathbb R$ satisfies the \emph{Cerami condition} $(C)$ if  every sequence $(u_j)_j \subset X$ such that
	\begin{equation}\label{cerami}
	(E(u_j))_j \mbox{ is bounded and }(1+\|u_j\|_{1,\mathcal H})E'(u_j)\to 0\mbox{ in } X^{*}\mbox{ as } j\rightarrow\infty
	\end{equation}
admits a convergent subsequence in $X$; see for instance \cite{BBF}. We say that $(u_j)_j$ is a Cerami sequence for $E$ if it satisfies \eqref{cerami}.

Let us now suppose that  $X$ is a reflexive and separable Banach space. It is well known that there exist $(e_j)_j \subset X$ and $(e^*_j)_j \subset X^*$ such that
	\begin{equation}
	\label{fountain1}
	X= \overline{\mathrm{span}\{e_j:\, j\in \mathbb{N}\}} \quad \text{as well as} \quad  X^*=\overline{\mathrm{span}\{e_j^*:\, j\in \mathbb{N}\}}
	\end{equation}
and 
	\begin{equation}
	\label{fountain2}
	\langle e_i^*, e_j\rangle=
	\begin{cases}
	  1 & \text{if } i=j, \\
	  0 & \text{if } i\neq j.
	\end{cases}
	\end{equation}
For all  $j\in\mathbb N$ we set
	\begin{equation}\label{3.15}
	X_j:=\mathrm{span}\{e_j\}, \qquad Y_j:=\bigoplus_{i=1}^j X_i, \qquad Z_j:=\bigoplus_{i=j}^{\infty} X_i.
	\end{equation}
We can then state the following  result, given in \cite[Theorem 2.9]{L}, which is a variant of the classical Fountain theorem \cite[Theorem 3.6]{Wil} for functionals that satisfy the Cerami condition instead of the Palais-Smale condition.

\begin{theorem}
\label{thm:fountain}

Let $E \in C^1(X, \RR)$ satisfy the Cerami condition $(C)$ and be such that $E(-u)= E(u)$. Moreover, suppose that for every $j \in \NN$ there exist $\rho_j> \gamma_j> 0$ such that
	\begin{itemize}
	
	\item[(i)] $\displaystyle b_j:= \inf_{u \in Z_j, \|u\|= \gamma_j} E(u) \to \infty $ as $j \to \infty$,
	
	\item[(ii)] $\displaystyle a_j:= \max_{u \in Y_j, \|u\|= \rho_j} E(u) \le 0$.
	
	\end{itemize}

Then, $E$ has a sequence of critical points $(u_j)_j$ such that $E(u_j) \to \infty$.
\end{theorem}

\section{Constant sign solutions}
\label{sec3}

In this section we prove the existence of constant-sign solutions to a class of  \eqref{P} and \eqref{P2} with superlinear nonlinearities. More specifically, we consider
	\begin{equation}
	\label{hp-h}
	\begin{aligned}
	h_1(x,t)&= (\zeta-\vartheta) |t|^{p-2}t- a(x) |t|^{q-2}t- f(x,t) \quad && \text{for a.e. } x \in \Om, \\
	h_2(x,t)&= -\beta |t|^{p-2}t && \text{for a.e. } x \in \rand,
	\end{aligned}
	\end{equation}
for all $t \in \RR$, where $\beta>0$, $\zeta>\vartheta>0$ are  parameters to be further specified, and $f$ is a Carath\'eodory function that satisfies suitable structure conditions stated below. Then problem \eqref{P} can be rewritten as
	\begin{equation}\label{problem1}
	\left\{
	\begin{aligned}
	& -M\left[\displaystyle\int_\Omega\left(\frac{|\nabla u|^p}{p}+a(x)\frac{|\nabla u|^q}{q}\right)dx\right]  \mbox{div}\left(|\nabla u|^{p-2}\nabla u+a(x)|\nabla u|^{q-2}\nabla u\right) \\
& \qquad = (\zeta- \vartheta) |u|^{p-2} u- a(x) |u|^{q-2} u- f(x, u) \quad && \mbox{in } \Omega,\\
	& M\left[\displaystyle\int_\Omega\left(\frac{|\nabla u|^p}{p}+a(x)\frac{|\nabla u|^q}{q}\right)dx\right]\left(|\nabla u|^{p-2}\nabla u+a(x)|\nabla u|^{q-2}\nabla u\right)\cdot \nu = -\beta |u|^{p-2} u && \mbox{on } \partial\Omega.
	\end{aligned}
	\right.
	\end{equation}
We assume that the nonlinearity $f\colon \Om \times \RR \to \RR $ is a  \textit{Carath\'eodory} function such that
		\begin{itemize}
		
		\item[($f_1$)] \textit{$f$ is bounded on bounded subsets of $\Om\times\RR$;}
		
		\item[($f_2$)] \textit{It holds
			\[
			\begin{aligned}
			& \lim_{t \to \pm \infty} \frac{f(x,t)}{ |t|^{q-2} t}= \infty \quad && \text{uniformly for a.e. } x \in \Om; \\
			\end{aligned}
			\]}
		
		\item[($f_3$)] \textit{It holds
			\[
			\begin{aligned}
			& \lim_{t \to 0} \frac{f(x, t)}{|t|^{p-2}t}= 0 \quad && \text{uniformly for a.e. } x \in \Om. \\
			\end{aligned}
			\]}
		
		\end{itemize}
A classical model for $f$ satisfying $(f_1)$-$(f_3)$ is given by $f(x,t):=w(x)|t|^{k-2}t$, where $k>q$ and $w\in\Linf$ with $\inf_\Om w>0$.
We say that $u \in W^{1, \mathcal H}(\Om)$ is a weak solution to \eqref{problem1} if it satisfies
	\[
	\begin{split}
	& M\l[\phi_{\mathcal H}(\nabla u)\r] \into \l(|\nabla u|^{p-2} \nabla u+ a(x) |\nabla u|^{q-2} \nabla u \r) \cdot \nabla \varphi \;dx+ \into \l(\vartheta |u|^{p-2}u+ a(x) |u|^{q-2}u \r)\varphi \;dx \\
	&= \into \l(\zeta |u|^{p-2} u-f(x, u)\r) \varphi \;dx- \beta \intor  |u|^{p-2} u  \varphi \;d\sigma
	\end{split}
	\]
for all $\varphi \in W^{1, \mathcal H}(\Om)$.

Our existence result for problem \eqref{problem1} reads as follows.

\begin{theorem}
\label{thm:constant-sign-robin1}
Let \eqref{condizione1}, $(M_1)$-$(M_2)$, and $(f_1)$-$(f_3)$ hold true. Moreover, assume that $\vartheta > 0$ and
\begin{equation}\label{nuovahp}
    \zeta> \vartheta+ \max\{1,M(0)\}\lambda_{1, p},
\end{equation}
where $\lambda_{1, p}$ is the first eigenvalue of the Robin eigenvalue problem given in \eqref{var-robin}. 
 Then, there exist two nontrivial weak solutions $\widetilde u$, $\underline u \in W^{1, \mathcal H}(\Om) \cap L^{\infty}(\Om)$ to problem \eqref{problem1}, such that $\widetilde u \ge 0 $ and $\underline u \le 0$.
\end{theorem}

\begin{proof}

First of all we observe that hypothesis $(f_2)$ allows us to find a constant $A= A(\zeta)>1$ such that
	\begin{equation}
	\label{1}
	\begin{aligned}
	& f(x,t)t \ge \zeta |t|^q \quad && \text{for a.e. } x \in \Om \text{ and all } |t| \ge A. \\
	\end{aligned}
	\end{equation}
We start with the existence of a nonnegative solution of \eqref{problem1}. From \eqref{1} we can take a constant function $u_0 \in (A, \infty)$  and use the fact that $p< q$ to achieve
	\begin{equation}
	\label{2}
	\zeta u_0^{p-1}- f(x, u_0) \le 0 \quad \text{for a.e. } x \in \Om.
	\end{equation}
Moreover we  consider the cut-off functions $b^{+}\colon \Om \times \RR \to \RR $ and $b_{\beta}^{+}\colon \rand \times \RR \to \RR$  given by
	\begin{equation}
	\label{cut-off}
	 b^+(x, t)=
	\begin{cases}
	0 & \text{if } t< 0 \\
	\zeta t^{p-1}- f(x, t) & \text{if } 0 \le t < u_0 \\
	\zeta u_0^{p-1}- f(x, u_0) & \text{if } t \ge u_0
	\end{cases}, 
	\qquad 
	b_{\beta}^+(x, t)=
	\begin{cases}
	0 & \text{if } t< 0 \\
	-\beta t^{p-1} & \text{if } 0 \le t < u_0 \\
	-\beta u_0^{p-1} & \text{if } t \ge u_0
	\end{cases},
	\end{equation}
and  set
	\begin{equation}
	\label{integr-funct}
	B^{+}(x, t):= \int_0^t b^{+}(x, s) \;ds \quad \text{as well as} \quad B_{\beta}^{+}(x, t):= \int_0^t b_{\beta}^{+}(x, s) \;ds, \quad \text{and} \quad F(x, t):= \int_0^t f(x, s) \;ds.
	\end{equation}
	It is easy to verify that the functions in \eqref{cut-off} and \eqref{integr-funct} are Carath\'eodory functions.
Moreover we consider the $C^1$-functional $J^{+}\colon W^{1, \mathcal H}(\Om) \to \RR $ defined by
	\begin{equation}
	\label{J}
	J^{+}(u)= \mathscr M\l[\phi_{\mathcal H}(\nabla u)\r]+ \frac{\vartheta}{p} \|u\|_p^p+ \frac{1}{q} \|u\|_{q, a}^q- \into B^{+}(x, u) \;dx- \intor B_{\beta}^{+}(x, u) \;d\sigma.
	\end{equation}
We aim to apply the direct methods of the calculus of variations to $J^+$. To this end, we first show that $J^+$ is  coercive. Indeed, let $u \in W^{1, \mathcal H}(\Om)$ be such that $\|u\|_{1, \Hi}>1$. Using hypothesis $(M_1)$ and taking into account that $p<q$ we have
	\begin{equation*}
	\label{coerc1}
	\mathscr M[\phi_{\mathcal H}(\nabla u)] \ge \frac{1}{\theta}M[\phi_{\mathcal H}(\nabla u)] \phi_{\mathcal H}(\nabla u) \ge \frac{1}{q\theta}  M[\phi_{\mathcal H}(\nabla u)] \varrho_{\mathcal H}(\nabla u).
	\end{equation*}
Furthermore, we see that, thanks to the truncations \eqref{cut-off} and hypothesis $(f_1)$, we have that the last two terms in \eqref{J} are bounded. 
These facts together with \eqref{stimaA1} yield that $J^+$ is coercive.

We then show that $J^+$ is also (sequentially weakly) lower semicontinuous. Indeed, let us take $u\in\soh$ and $(u_j)_j\subset \soh$ such that $u_j\rightharpoonup u$ in $\soh$. 
By Propositions \ref{P2.1}-\ref{P2.3} and \cite[Theorem 4.9]{B} there exists a subsequence, still denoted by $(u_j)_j$, such that, as $j\to\infty$,
	\begin{equation*}\label{convergenzex}
	\begin{gathered}
	u_j\to u\mbox{ in }L^{p}(\Omega)\cap L^q_a(\Omega),\qquad \nabla u_j\rightharpoonup\nabla u\mbox{ in }\left[L^{\mathcal H}(\Omega)\right]^N,\qquad\phi_\mathcal{H}(\nabla u_j)\to\ell,\\
	u_j(x)\to u(x)\mbox{ for a.e. }x\in\Omega,\qquad u_j(x)\to u(x)\mbox{ for $\sigma$-a.e. }x\in\partial\Omega.
	\end{gathered}
	\end{equation*}
We use the fact that $\mathscr{M}$ is increasing and continuous,  $\phi_\mathcal{H}$ is  sequentially weakly lower semicontinuous, and  $B^+$ and $B^+_\beta$ are Carath\'eodory functions bounded from below to apply the Fatou's lemma and achieve
	  \begin{equation*}
    \begin{split}
        J^+(u)&\leq\mathscr{M}\left[\lim_{j\to\infty}\phi_\mathcal{H}(\nabla u_j)\right]+\liminf_{j\to\infty}\left(\frac{\vartheta}{p}\|u_j\|^p_p+\frac{1}{q}\|u_j\|^q_{q,a}\right) \\
        & \qquad \quad -\into\lim_{j\to\infty} B(x,u_j) \; dx- \intor\lim_{j\to\infty} B^+_\beta(x,u_j) \;d\sigma\\
        &\leq\lim_{j\to\infty}\mathscr{M}[\phi_\mathcal{H}(\nabla u_j)]+\liminf_{j\to\infty}\left(\frac{\vartheta}{p}\|u_j\|^p_p+\frac{1}{q}\|u_j\|^q_{q,a}\right)- \liminf_{j\to\infty} \into B(x,u_j)\; dx \\
        & \qquad \quad -\liminf_{j\to\infty}\intor B^+_\beta(x,u_j) \; d\sigma\\
        &\leq\liminf_{j\to\infty}J^+(u_j).
        \end{split}
	\end{equation*}
Therefore, there exists a function $\widetilde u \in W^{1, \mathcal H}(\Om)$ such that
	\begin{equation}
	\label{inf}
	J^+(\widetilde u)= \inf \l\{J^+(u): \, u \in W^{1, \mathcal H}(\Om)\r\}.
	\end{equation}
Let us verify that $\widetilde u$ is not trivial. First of all, thanks to hypothesis $(f_3)$, for all $\eps>0$ there exists $\delta \in (0,1)$ such that
	\begin{equation}
	\label{3}
	F(x, t) \le \frac{\eps}{p} |t|^p \quad \text{for a.e. } x \in \Om \text{ and for all } |t| \le \delta.
	\end{equation}
We further define the function $h\colon [0,\infty)\to[M(0),\infty) $ such that
\begin{equation}\label{acca}
		h(t)=
		\begin{cases}\,M(0)&\text{ if }t=0\\
			\displaystyle\max_{s\in(0,t]} \frac{\mathscr{M}(s)}{s} &\text{ otherwise}
		\end{cases}
		.
	\end{equation}
It is easy to see that $h$ is well-defined and continuous. These facts, together with \eqref{nuovahp}, guarantee that there exists $t_1\in(0,1)$ such that
	\begin{equation}\label{t1}
    \zeta> \vartheta+ \max\{1,h(t_1)\}\lambda_{1, p}.
	\end{equation}
Let $u_{1, p}$ be the normalized eigenfunction, that is $\|u_{1,p}\|_p= 1$, corresponding to $\lambda_{1,p}$. Since $u_{1,p} \in \text{int} \l(C^1(\close)_+\r)$ we can choose $t_2 \in (0,t_1]$ sufficiently small so that $ t u_{1,p} \in [0, \delta]$ for all $t\in(0,t_2]$, which implies that 
	\begin{equation}\label{t2}
    t u_{1,p} \in (0, u_0)\qquad\text{for all }x \in \close\text{ and }t\in(0,t_2].
	\end{equation}
Moreover, let us choose $t_3\in(0,t_2]$ such that 
	\begin{equation}\label{t3}
    \phi_\mathcal{H}(\nabla(tu_{1,p}))\leq t_1\qquad\text{for all $t\in[0,t_3]$}.
	\end{equation}
Taking \eqref{acca} into account, \eqref{t3} implies that
	\[
	\mathscr M[\phi_{\Hi}(\nabla (t u_{1,p}))] \le h(t_1) \phi_{\Hi}(\nabla (t u_{1,p}))
	\]
for every $t\in(0,t_3]$.
Then we use the truncations in \eqref{cut-off}, together with \eqref{3}, \eqref{t2}, \eqref{t3}, and $(f_2)$ to achieve
	\begin{equation}
	\label{4}
	\begin{split}
	J^+({t} u_{1,p})
	&= \mathscr M \l[\phi_{\mathcal H}(\nabla({t} u_{1,p}))\r]+ \frac{\vartheta}{p} \|{t} u_{1,p}\|_p^p+ \frac{1}{q} \|{t} u_{1,p}\|_{q, a}^q \\
	& \qquad - \into B^+(x, {t} u_{1,p}) \; dx- \intor B_{\beta}^+(x, {t} u_{1,p})\; d\sigma \\
	& \le h(t_1) \frac{{t}^p}{p} \|\nabla  u_{1,p} \|_p^p+ h(t_1) \frac{{t}^q}{q} \|\nabla u_{1, p}\|_{q, a}^q+ \vartheta \frac{ {t}^p}{p}+ \frac{{t}^q}{q} \|u_{1,p}\|_{q, a}^q \\
	& \qquad -\zeta \frac{{t}^p}{p} + \into F(x, {t} u_{1,p}) \; dx+ \beta \frac{{t}^p}{p} \|u_{1,p}\|_{p, \rand}^p \\
	& \le \max\l\{1, h(t_1) \r\} \frac{{t}^p}{p} \l(\|\nabla u_{1,p}\|_p^p+ \beta \|u_{1,p}\|_{p, \rand}^p \r)+ h(t_1) \frac{{t}^q}{q} \|\nabla u_{1,p}\|_{q, a}^q \\
	& \qquad + \vartheta \frac{{t}^p}{p}+ \frac{{t}^q}{q} \|u_{1,p}\|_{q, a}^q- \zeta \frac{{t}^p}{p}+ \eps \frac{{t}^p}{p} \\
	& \le {t}^p \l(\frac{\max\l\{1, h(t_1) \r\} \lambda_{1,p}+ \vartheta- \zeta+ \eps}{p} \r)+ {t}^q \l(\frac{h(t_1) \|\nabla u_{1,p}\|_{q, a}^q+ \|u_{1,p}\|_{q, a}^q}{q} \r),
	\end{split}
	\end{equation}
for every $t\in (0,t_3]$. Since $p<q$, we can choose $\overline{t}\in(0,t_3]$ such that
\begin{equation}\label{t4}
    -\overline{t}^p \l(\frac{\zeta-\vartheta-\max\l\{1, h(t_1) \r\} \lambda_{1,p}-\eps}{p} \r)+ \overline{t}^q \l(\frac{h(t_1) \|\nabla u_{1,p}\|_{q, a}^q+ \|u_{1,p}\|_{q, a}^q}{q} \r)<0,
\end{equation}
where $\varepsilon:=\frac{1}{2} ({\zeta-\vartheta-\max\l\{1, h(t_1) \r\} \lambda_{1,p}})>0$  thanks to \eqref{t1}.

From \eqref{4} it follows that $J^+(\overline{t} u_{1,p})< 0$, therefore, by \eqref{inf} we have
	\[
	J^+(\widetilde u)\leq J^+(tu_{1,p})<0= J^+(0),
	\]
which implies that $\widetilde u \ne 0 $. Let us now show the bound $\widetilde u \in [0, u_0]$. We observe that \eqref{inf} implies that $\l(J^+\r)'(\widetilde u)= 0$, that is,
	\begin{equation}
	\label{5}
	\begin{split}
	&  M\l[\phi_{\mathcal H}(\nabla \widetilde u)\r] \into \l(|\nabla \widetilde u|^{p-2} \nabla \widetilde u+ a(x) |\nabla \widetilde u|^{q-2} \nabla \widetilde u \r) \cdot \nabla \varphi \; dx  + \into \l(\vartheta |\widetilde u|^{p-2} \widetilde u+ a(x) |\widetilde u|^{q-2} \widetilde u \r) \varphi \; dx \\
	&= \into b^+(x, \widetilde u) \varphi \; dx+ \intor b_{\beta}^+(x, \widetilde u) \varphi\; d\sigma
	\end{split}
	\end{equation}
for all $\varphi \in W^{1, \mathcal H}(\Om)$.  Choosing $\varphi= -\widetilde u^- \in W^{1, \mathcal H}(\Om)$ in \eqref{5} and taking into account the truncations \eqref{cut-off}, we get
	\begin{equation}
	\label{6}
	 M\l[\phi_{\mathcal H}(\nabla \widetilde u)\r] \l(\|\nabla \widetilde u^-\|_p^p+ \|\nabla \widetilde u^-\|_{q, a}^q \r)+ \|\widetilde u^- \|_p^p+ \|\widetilde u^-\|_{q, a}^q= 0.
	\end{equation}
Since all the above terms are nonnegative, the equality in \eqref{6} is satisfied when $\widetilde u^-= 0$, which implies that $\widetilde u \ge 0 $. 

We now choose $\varphi= (\widetilde u- u_0)^+ \in W^{1, \mathcal H}(\Om)$ in \eqref{5} and take \eqref{cut-off} into account once again to achieve
	\begin{equation}
	\label{7}
	\begin{split}
	&  M\l[\phi_{\mathcal H}(\nabla \widetilde u)\r] \into \l(|\nabla \widetilde u|^{p-2} \nabla \widetilde u+ a(x) |\nabla \widetilde u|^{q-2} \nabla \widetilde u \r) \cdot \nabla (\widetilde u- u_0)^+ \; dx \\
	& \qquad + \into \l(\vartheta \widetilde u^{p-1}+ a(x) \widetilde u^{q-1} \r) (\widetilde u- u_0)^+ \; dx \\
	&= \into b^+(x, \widetilde u) (\widetilde u- u_0)^+ \; dx+ \intor b_{\beta}^+(x, \widetilde u) (\widetilde u- u_0)^+ \; d\sigma \\
	&= \into (\zeta u_0^{p-1}- f(x, u_0)) (\widetilde u- u_0)^+ \; dx+ \intor (-\beta u_0^{p-1}) (\widetilde u- u_0)^+ \; d\sigma \\
	& \le 0,
	\end{split}
	\end{equation}
where the last inequality holds by \eqref{2} and since $u_0>0$. On the one hand, we see that
	\begin{equation}\label{8}
	\begin{split}
	&  M\l[\phi_{\mathcal H}(\nabla \widetilde u)\r] \into \l(|\nabla \widetilde u|^{p-2} \nabla \widetilde u+ a(x) |\nabla \widetilde u|^{q-2} \nabla \widetilde u \r) \cdot \nabla (\widetilde u- u_0)^+ \; dx \\
	& =M\l[\phi_{\mathcal H}(\nabla \widetilde u)\r] \int_{\{\widetilde u> u_0\}} |\nabla \widetilde u|^p+  a(x) |\nabla \widetilde u|^q \; dx, \\
	\end{split}
	\end{equation}
while on the other hand, exploiting the fact that $\widetilde u> u_0> 1$, we get
	\begin{equation}\label{9}
	\begin{split}
	0 & \ge \into (\vartheta \widetilde u^{p-1}+ a(x) \widetilde u^{q-1}) (\widetilde u- u_0)^+ dx \\
	&= \int_{\{\widetilde u> u_0\}} \vartheta \widetilde u^{p-1} (\widetilde u- u_0)+ a(x) \widetilde u^{q-1} (\widetilde u- u_0) \; dx \\
	& \ge \int_{\{\widetilde u> u_0\}} \vartheta  (\widetilde u- u_0)^p+ a(x)  (\widetilde u- u_0)^q \; dx.
	\end{split}
	\end{equation}
Gathering \eqref{7}, \eqref{8} and \eqref{9}, we see that
	\[
	\begin{split}
	 M\l[\phi_{\mathcal H}(\nabla \widetilde u)\r] & \int_{\{\widetilde u> u_0\}} |\nabla \widetilde u|^p+ a(x) |\nabla \widetilde u|^q \; dx+  \int_{\{\widetilde u> u_0\}} \vartheta (\widetilde u- u_0)^p a(x) (\widetilde u- u_0)^q \; dx \leq 0,
	\end{split}
	\]
which gives $(\widetilde u- u_0)^+= 0$, and therefore $\widetilde u \le u_0$. This fact in particular implies that $\widetilde u \in L^{\infty}(\Om)$. 
By definition of the truncations \eqref{cut-off}, $\widetilde u$ turns out to be a weak solution to \eqref{problem1}.

In order to show the existence of a nonpositive solution, we first fix the constant function $u_1 \equiv -u_0$ and use equation \eqref{1} and the fact that $p< q$ to achieve
	\[
	\zeta |u_1|^{p-2} u_1- f(x, u_1) \ge 0 \quad \text{for a.e. } x \in \Om.
	\]
Then we consider the cut-off, Carath\'eodory functions $b^-\colon \Om \times \R \to \R $ and $b^-_{\beta}\colon \rand \times \R \to \R$ given by
	\begin{equation}
	\label{b-meno}
	b^-(x, t)=
	\begin{cases}
	\zeta |u_1|^{p-2} u_1- f(x, u_1) & \text{if } t \le u_1 \\
	\zeta |t|^{p-2} t- f(x, t) & \text{if } u_1 < t \le 0 \\
	0 & \text{if } t> 0
	\end{cases}, \qquad
	 b_{\beta}^-(x, t)=
	\begin{cases}
	-\beta |u_1|^{p-2} u_1 & \text{if } t \le u_1 \\
	-\beta |t|^{p-2} t & \text{if } u_1 < t \le 0 \\
	0 & \text{if } t>0
	\end{cases},
	\end{equation}
set 
	\begin{equation}
	\label{b-meno1}
	B^-(x, t):= \int_0^t b^-(x, s) \; ds \quad \text{as well as} \quad B_{\beta}^-(x, t):= \int_0^t b_{\beta}^-(x, s) \; ds,
	\end{equation}
and consider the $C^1$-functional $J^-\colon W^{1, \mathcal H}(\Om) \to \R $ given by
	\[
	J^-(u)= \mathscr M[\phi_{\mathcal H}(\nabla u)]+ \frac{\vartheta}{p} \|u\|_p^p+ \frac{1}{q} \|u\|_{q, a}^q- \into B^-(x, u) \; dx- \intor B_{\beta}^-(x, u) \; d\sigma.
	\]
Reasoning in a similar fashion as before we find a global minimizer $\underline u \in W^{1, \mathcal H}(\Om)$ of $J^-$ such that $\underline u \in [u_1, 0]$. 
By definition of the truncations \eqref{b-meno} we see that $\underline u$ is a nonpositive weak solution of problem \eqref{problem1}.
The proof is complete.
\end{proof}

\medskip

Now, we consider problem \eqref{P2} with $h_1$, $h_2$ as in \eqref{hp-h}. Then we get
	\begin{equation}
	\label{problem1.2}
	\left\{
	\begin{aligned}
	& -M\l(\into |\nabla u|^p dx \r) \Delta_p u- M\l(\into a(x) |\nabla u|^q dx \r) \divergenz \l(a(x) |\nabla u|^{q-2} \nabla u\r) \\
	& \qquad \qquad= (\zeta- \vartheta) |u|^{p-2} u- a(x) |u|^{q-2} u- f(x, u) \qquad && \text{in } \Om, \\
	& \left[M(\Vert\nabla u\Vert^p_p)|\nabla u|^{p-2}\nabla u+M(\Vert\nabla u\Vert^q_{q,a})a(x)|\nabla u|^{q-2}\nabla u\right]\cdot \nu = \beta |u|^{p-2} u \qquad&& \mbox{on } \partial\Omega.
	\end{aligned}
	\right.
	\end{equation}
We say that $u \in W^{1, \cH}(\Om)$ is a weak solution to \eqref{problem1.2} if 
	\[
	\begin{split}
	& M\l(\|\nabla u\|_p^p\r) \into |\nabla u|^{p-2} \nabla u \cdot \nabla \varphi \; dx+ M\l(\|\nabla u\|_{q, a}^q\r) \into a(x) |\nabla u|^{q-2} \nabla u \cdot \nabla \varphi \; dx \\
	& \quad + \into \l(\vartheta |u|^{p-2} u+ a(x) |u|^{q-2} u \r) dx = \into \l(\vartheta |u|^{p-2} u- f(x, u)\r) \varphi \; dx- \beta \intor |u|^{p-2} u \varphi \; d\sigma
	\end{split}
	\]
holds for all $ \varphi \in W^{1, \cH}(\Om)$. The existence result  concerning problem \eqref{problem1.2} reads as follows. 

\begin{theorem}
\label{thm:constant-sign-robin2}
In the same hypotheses of Theorem \ref{thm:constant-sign-robin1}, there exist two nontrivial weak solutions $\widetilde u$, $\underline u \in W^{1, \cH}(\Om) \cap L^{\infty}(\Om)$ to problem \eqref{problem1.2}, such that $ \widetilde u \ge 0 $ and $\underline u \le 0$. 
\end{theorem}

\begin{proof}

As before, we start with the nonnegative solution, following the proof of Theorem \ref{thm:constant-sign-robin1} until \eqref{integr-funct}. Then we consider the $C^1$-functional $\mathcal J^+\colon W^{1, \cH}(\Om) \to \R$ given by
	\begin{equation*}
	\label{J1}
	\begin{split}
	\mathcal J^+(u)&= \frac{1}{p} \mathscr M \l(\|\nabla u\|_p^p\r)+ \frac{1}{q} \mathscr M\l(\|\nabla u\|_{q, a}^q\r)+ \frac{\vartheta}{p} \|u\|_p^p+ \frac{1}{q} \|u\|_{q, a}^q-\into B^+(x, u) \; dx- \intor B_{\beta}^+(x, u) \; d\sigma.
	\end{split}
	\end{equation*}
	Arguing as in the previous proof and thanks to \eqref{stimaA2} we have  that $\mathcal J^+$ is  coercive and lower semicontinuous.
Therefore, there exists $\widetilde u \in W^{1, \mathcal H}(\Om)$ such that
	\begin{equation*}
	\label{inf2}
	\mathcal J^+(\widetilde u)= \inf\{\mathcal J^+(u): \, u \in W^{1, \cH}(\Om)\}.
	\end{equation*}
In order to show that  $\widetilde u$ is nontrivial, we let $u_{1,p} \in \text{int} \l(C^1(\close)_+\r)$  be the normalized eigenfunction corresponding to the first eigenvalue $\lambda_{1,p}$ of problem \eqref{robin}, and choose $t_1$, $t_2$ as in \eqref{t1} and \eqref{t2},  respectively. Moreover, we choose $t_3\in(0,t_2]$ such that
	\begin{equation*}\label{t3mod}
	\|\nabla (t u_{1,p}) \|_p^p,\, \|\nabla (t u_{1,p})\|_{q, a}^q\leq t_1\qquad\text{for all $t\in[0,t_3]$},
	\end{equation*}
and finally  $\overline{t}$ and $\eps>0$ in order to satisfy \eqref{t4}. It follows that 
	\[
	\mathcal J^+(\widetilde u) \le \mathcal J^+(\overline{t} u_{1,p})< 0=\mathcal J^+(0),
	\]
 therefore $\widetilde u \not\equiv 0$. The proof that  $ 0 \le \widetilde u \le u_0$ works as in the previous case, and thus $\widetilde u$ is a weak, nonnegative, and bounded solution to  \eqref{problem1.2}. 

In order to find the nonpositive solution we consider the functional $\mathcal J^-\colon W^{1, \cH}(\Om) \to \R$ defined by
	\[
	\begin{split}
	\mathcal J^-(u)&= \frac{1}{p} \mathscr M \l(\|\nabla u\|_p^p\r)+ \frac{1}{q} \mathscr M\l(\|\nabla u\|_{q, a}^q\r)+ \frac{\vartheta}{p} \|u\|_p^p+ \frac{1}{q} \|u\|_{q, a}^q-\into B^-(x, u) \; dx- \intor B_{\beta}^-(x, u) \; d\sigma,
	\end{split}
	\]
with $B^-$ and $B_{\beta}^-$ given as in \eqref{b-meno}, \eqref{b-meno1}, and reason as in the previous case. This completes the proof of the theorem. 
\end{proof}

We conclude this section pointing out that Theorems \ref{thm:constant-sign-robin1} and \ref{thm:constant-sign-robin2} generalize \cite[Theorem 4.2]{EMW} in a nonlocal Kirchhoff framework. In any case, assumption \eqref{nuovahp} is consistent with the constraint for $\zeta$ requested in \cite{EMW}, when $M\equiv1$.

\section{Infinitely many solutions} \label{sec4}

In this section we prove the existence of infinitely many solutions to \eqref{P} and \eqref{P2} when $h_1$ and $h_2$ are symmetric. Throughout this section and differently from the previous one, $\partial\Om$ is assumed to be only Lipschitz. More precisely, we choose
	\begin{equation}\label{varieh}
	\begin{aligned}
	h_1(x,t)&=f(x,t)- |t|^{p-2} t- a(x) |t|^{q-2} t \quad && \text{for a.e. } x \in \Om,\,\text{for all }t\in\RR, \\
	h_2(x, t)&= g(x,t) && \text{for a.e. } x \in \rand,\,\text{for all }t\in\RR,
	\end{aligned}
	\end{equation}
where $f:\Omega\times\mathbb R\to\mathbb R$ and $g:\partial\Omega\times\mathbb R\to\mathbb R$ are two \textit{Carath\'{e}odory} functions that satisfy the following assumptions:

\vspace{0.1cm}
\begin{enumerate}
\item[$(h_1)$] {\em there exist exponents $r_1\in (p,p^*)$ and $r_2\in (p,p_*)$, and two constants $c_1$, $c_2>0$ such that
			\[
			\begin{aligned}
			& |f(x,t)|\leq c_1\left(1+|t|^{r_1-1}\right) \quad && \text{for a.e. } x \in \Om  \text{ and all }t\in\RR, \\
			& |g(x,t)|\leq c_2\left(1+|t|^{r_2-1}\right) \quad && \text{for a.e. } x \in \rand  \text{ and all }t\in\RR; \\
			\end{aligned}
			\]
}
\item[$(h_2)$]
{\em $F(x,t)\geq0$ and $G(x,t)\geq0$ for a.e. $x\in\Omega$ and $x\in\partial\Omega$, respectively, and all $t\in\RR$,} where
$$F(x,t):=\int_0^tf(x,s)\;ds \quad \text{as well as} \quad G(x,t):=\int_0^tg(x,s)\;ds. $$ 
Moreover it holds that
			\[
			\begin{aligned}
			& \lim_{|t|\to\infty}\frac{F(x,t)}{|t|^{q\theta}}=\infty \quad && \text{uniformly for a.e. } x \in \Om. \\
			\end{aligned}
			\]
\item[$(h_3)$] {\em It holds that
$$\mathcal F(x,t):=\frac{1}{q\theta}f(x,t) t-F(x,t)\geq 0 \quad \text{as well as} \quad \mathcal G(x,t):=\frac{1}{q\theta}g(x,t) t-G(x,t)\geq 0$$
for a.e. $x\in\Om$ and a.e. $x \in \rand$, respectively, and for every $t\in\RR$;}
\item[$(h_4)$]
{\em There exist $t_0>0$, $d_1$, $d_2>0$, $s_1>N/p$ and $s_2>(N-1)/(p-1)$ such that
			\[
			\begin{aligned}
			& F(x,t)^{s_1}\leq d_1|t|^{ps_1}\mathcal F(x,t) \quad && \text{for a.e. } x \in \Om,\,\text{ for all }|t|>t_0, \\
			& G(x,t)^{s_2}\leq d_2|t|^{ps_2}\mathcal G(x,t) \quad && \text{for a.e. } x \in \rand,\,\text{ for all }|t|>t_0, \\
			\end{aligned}
			\]
}
\item[$(h_5)$]
{\em $f(x,-t)=-f(x,t)$ and $g(x,-t)=-g(x,t)$ for a.e. $x\in\Om$ and a.e. $x\in\rand$, respectively, and all }$t\in\mathbb R$.
\end{enumerate}
\vspace{0.1cm}
\begin{remark}
When $\theta<p_*/q$, two simple model functions $f$ and $g$ satisfying $(h_1)$--$(h_5)$ are given by $f(x,t)=w(x)|t|^{r_1-2}t$ and $g(x,t)=z(x)|t|^{r_2-2}t$, where $w\in\Linf$ with $\inf_\Om w>0$ and $z\in L^\infty(\partial\Om)$ with $\inf_{\partial\Om}z>0$, and with parameters $t_0>0$ and
\begin{equation}\label{valori}
\begin{gathered}
s_1\in\left(\frac{N}{p},\frac{q\theta}{q\theta-p}\right),\qquad r_1\in\left(q\theta,\frac{ps_1}{s_1-1}\right),\qquad d_1=\frac{1}{r_1^{s_1}}\cdot\frac{\|w\|_\infty^{s_1}}{\inf_\Om w}\cdot\frac{r_1q\theta}{r_1-q\theta},\\
s_2\in\left(\frac{N-1}{p-1},\frac{q\theta}{q\theta-p}\right),\qquad r_2\in\left(q\theta,\frac{ps_2}{s_2-1}\right), \qquad d_2=\frac{1}{r_2^{s_2}}\cdot\frac{\|z\|_{\infty,\partial\Om}^{s_2}}{\inf_{\partial\Om} z}\cdot\frac{r_2q\theta}{r_2-q\theta}.
\end{gathered}
\end{equation}
We observe that such parameters in \eqref{valori} exist since $\theta<p_*/q$. Moreover, it can be easily seen that this restriction on the choice of $\theta$ is necessary if we want to have models of this form.
\end{remark}

With the choice \eqref{varieh} we have that problem \eqref{P} can be rewritten as
	\begin{equation}\label{problem2}
	\left\{
	\begin{aligned}
	& -M\left[\phi_{\mathcal H}(\nabla u)\right]  \mbox{div}\left(|\nabla u|^{p-2}\nabla u+a(x)|\nabla u|^{q-2}\nabla u\right) 
	= f(x, u) - |u|^{p-2} u- a(x) |u|^{q-2} u \quad && \mbox{in } \Omega,\\
	& M\left[\phi_{\mathcal H}(\nabla u)\right]\left(|\nabla u|^{p-2}\nabla u+a(x)|\nabla u|^{q-2}\nabla u\right)\cdot \nu = g(x,u) && \mbox{on } \partial\Omega.
	\end{aligned}
	\right.
	\end{equation}
We  say that $u \in W^{1, \mathcal H}(\Om)$ is a weak solution to \eqref{problem2} if it satisfies
	\[
	\begin{split}
	& M\l[\phi_{\mathcal H}(\nabla u)\r] \into \l(|\nabla u|^{p-2} \nabla u+ a(x) |\nabla u|^{q-2} \nabla u \r) \cdot \nabla \varphi \;dx+ \into \l(|u|^{p-2}u+ a(x) |u|^{q-2}u \r)\varphi \;dx \\
	&= \into f(x, u) \varphi \;dx+ \intor  g(x,u)  \varphi \;d\sigma
	\end{split}
	\]
for all $\varphi \in W^{1, \mathcal H}(\Om)$. 
We see that \eqref{problem2} is the Euler-Lagrange equation corresponding to  the functional $I:\soh\to\RR$ of class $C^1$ defined as
$$
I(u):= \mathscr M\l[\phi_{\mathcal H}(\nabla u)\r]+ \frac{1}{p} \|u\|_p^p+ \frac{1}{q} \|u\|_{q,a}^q- \into F(x, u) \;dx- \intor G(x, u) \;d\sigma.
$$
Now, we are ready to state the existence result to \eqref{problem2}.
\begin{theorem}\label{thm:infinitely}
Let \eqref{condizione1}, $(M_1)$-$(M_2)$, and $(h_1)$-$(h_5)$ hold true.
Then, problem \eqref{problem2} has infinitely many weak solutions $(u_j)_j$ with unbounded energy.
\end{theorem}

Theorem \ref{thm:infinitely} generalizes \cite[Theorem 5.9]{GW1} where problem \eqref{problem2} was considered with $M\equiv1$, namely without Kirchhoff coefficient. Here, the function $f$ satisfies the crucial assumption $(h_4)$, different from the quasi-monotonic assumption \eqref{ipotesi2} assumed in \cite{GW1}. Indeed, \eqref{ipotesi2} can not work for problem \eqref{problem2}, even if $M$ satisfies some monotonic condition. This is due to the fact that $M$ in \eqref{problem2} depends on $\phi_{\mathcal H}(\nabla u)$ given in \eqref{phi-H} and not on the seminorm $\|\nabla u\|_{\mathcal H}$. Also, we do not have a proper equivalence  between $\phi_{\mathcal H}(\nabla u)$ and $\|\nabla u\|_{\mathcal H}$, but only the following relation
$$
\frac{1}{q}\min\{\|\nabla u\|_{\mathcal H}^p,\|\nabla u\|_{\mathcal H}^q\}\leq\frac{1}{q}\varrho_{\mathcal H}(\nabla u)\leq
\phi_{\mathcal H}(\nabla u)
\leq\frac{1}{p}\varrho_{\mathcal H}(\nabla u)\leq\frac{1}{p}\max\{\|\nabla u\|_{\mathcal H}^p,\|\nabla u\|_{\mathcal H}^q\}
$$
implied by Proposition \ref{P2.1}.

We aim to prove Theorem \ref{thm:infinitely} by means of Theorem \ref{thm:fountain}. To this end, we first show that $I$ satisfies the $(C)$ condition and then that assumptions (i) and (ii) of Theorem \ref{thm:fountain} are satisfied. We start with the following preliminary result.

\begin{lemma}\label{lemma4.1}
Any Cerami sequence $(u_j)_j$ of $I$ is bounded in $W^{1,\mathcal H}(\Omega)$.
\end{lemma}

\begin{proof}
Let $(u_j)_j$ be a sequence satisfying \eqref{cerami} with $E=I$.
Hence, there exist $C>0$ and $\varepsilon_j>0$, with $\varepsilon_j\to 0$, such that
	\begin{equation}\label{4.1}
	|\langle I'(u_j),\varphi\rangle|\leq \frac{\varepsilon_j\|\varphi\|_{1,\mathcal H}}{1+\|u_j\|_{1,\mathcal H}}\quad\mbox{for all } \varphi\in W^{1,\mathcal H}(\Omega)\mbox{ and }j\in\mathbb{N}
	\end{equation}
and
	\begin{equation}
	\label{ineq-I}
	|I(u_j)|\leq C\quad\mbox{for all }j\in \mathbb{N}.
	\end{equation}
We claim  that $(u_j)_j$ is bounded in $W^{1,\mathcal H}(\Omega)$.

Arguing by contradiction, we assume that~$\|u_j\|_{1,\mathcal H}\to\infty$ as $j\to\infty$
and, without loss of generality, that $\|u_j\|_{1,\mathcal H}> 1$ for $j$ sufficiently large.
Set~$v_j:=u_j/\|u_j\|_{1,\mathcal H}$. It holds that $\|v_j\|_{1,\mathcal H}=1$, therefore there exists  $v\in W^{1,\mathcal H}(\Omega)$ such that, up to a subsequence,
	\begin{align}\label{4.3}
	v_j(x)\to v(x)\mbox{ for a.e. }x\in\Om,\qquad v_j\to v\mbox{ in }L^{\nu_1}(\Omega)\cap L^{\nu_2}(\partial\Omega)
	\end{align}
for all $\nu_1\in[1,p^*)$ and all $\nu_2\in[1,p_*)$, exploiting the reflexivity of $W^{1,\mathcal H}(\Omega)$ and Proposition \ref{P2.3}-(ii) and (iii).
We aim to show that $v= 0$. To this end, we set $\Omega^*=\{x\in\Omega:\,v(x)\neq0\}$ and assume that $|\Omega^*|>0$. Since $\|u_j\|_{1,\mathcal H}\to\infty$ as $j\to\infty$, then
	\[
	|u_j|=\big\|u_j\big\|_{1,\mathcal H}\cdot|v_j|\to \infty \quad  \mbox{a.e. in }  \Omega^*.
	\]
Taking into account $(h_2)$, we get
	\[
	\begin{aligned}
	&\infty= \lim_{j\to\infty}\frac{F(x,u_j)}{\|u_j\|_{1,\mathcal H}^{q\theta}}=\lim_{j\to\infty}\frac{F(x,u_j)}{|u_j|^{q\theta}}\cdot |v_j|^{q\theta}  \quad \text{for a.e. } x\in\Om^*.
	\end{aligned}
	\]
Then, Fatou's lemma gives
	\begin{align}\label{4.4}
	\lim_{j\to\infty}\int_{\Omega}\frac{F(x,u_j)}{\|u_j\|_{1,\mathcal H}^{ q\theta}}\;dx=\lim_{j\to\infty}\int_{\Omega}\frac{F(x,u_j)|v_j|^{ q\theta}}{|u_j|^{ q\theta}}\;dx=\infty.
	\end{align}
On the other hand, we use \eqref{cerami} together with the nonnegativity of $G$ first and then \eqref{stimaB1} to achieve
	\begin{equation*}
	\label{4.5}
	\begin{aligned}
	\int_{\Omega}F(x,u_j)\;dx & \leq \mathscr M[\phi_{\mathcal H}(\nabla u_j)]+\frac{1}{p}\|u_j\|_p^p+\frac{1}{q}\|u_j\|_{q,a}^q+C \\
	& \le {B_1}(1+\|u\|_{1, \cH}^q+\|u\|_{1, \cH}^{q\theta})+C \quad \text{ for all } j\in\mathbb{N}.
	\end{aligned}
	\end{equation*}
Dividing by $\|u_j\|_{1,\mathcal H}^{q\theta}$, passing to the limit superior as $j \to \infty$ and recalling that $\theta\geq 1$, we have
	\begin{equation}\label{differenza1}
    \limsup_{j\to\infty} \int_{\Omega}\frac{F(x,u_j)}{\|u_j\|_{1,\mathcal H}^{q\theta}}\;dx<\infty,
	\end{equation}
which contradicts \eqref{4.4}. 
In conclusion, $\Omega^*$ has zero measure, that is, $v=0$ a.e. in $\Omega$.

Take now $j \in \NN$ large enough so that $\|u_j\|_{1,\mathcal H}>1$. By hypothesis $(M_1)$, inequality \eqref{cerami}, together with \eqref{stimaA1} and the fact that  $\theta\geq 1$, we have
	\[
	\begin{aligned}
	C\geq I(u_j)=
	&\mathscr M[\phi_{\mathcal H}(\nabla u_j)]+\frac{1}{p}\|u_j\|_p^p+\frac{1}{q}\|u_j\|_{q,a}^q
-\int_{\Omega}F(x,u_j)\;dx-\int_{\partial\Omega}G(x,u_j)\;d\sigma\\
	\geq&\frac{1}{\theta}M[\phi_{\mathcal H}(\nabla u_j)]\phi_{\mathcal H}(\nabla u_j)+\frac{1}{q}\varrho_{\mathcal H}(u_j)
-\int_{\Omega}F(x,u_j)\;dx-\int_{\partial\Omega}G(x,u_j)\;d\sigma\\
	\geq&\frac{{A_1}}{q\theta}\|u_j\|_{1,\mathcal H}^p-\int_{\Omega}F(x,u_j)\;dx-\int_{\partial\Omega}G(x,u_j) \;d\sigma.
	\end{aligned}
	\]
Since $\|u_j\|_{1,\mathcal H}\to\infty$ as $j\to\infty$, from the previous inequality we have
	\begin{equation}\label{4.6}
	\liminf_{j\to\infty}\left[\int_{\Omega}\frac{F(x,u_j)}{\|u_j\|_{1,\mathcal H}^p}dx
+\int_{\partial\Omega}\frac{G(x,u_j)}{\|u_j\|_{1,\mathcal H}^p}d\sigma\right]\geq\frac{{A_1}}{q\theta}>0.
	\end{equation}
We aim to show that
	\begin{equation}
	\label{stima-FG}
	\lim_{j\to\infty}\int_\Omega\frac{F(x,u_j)}{\|u_j\|_{1,\mathcal H}^p}dx=
	\lim_{j\to\infty}\int_{\partial\Omega}\frac{G(x,u_j)}{\|u_j\|_{1,\mathcal H}^p}d\sigma=0,
	\end{equation}
which will eventually  contradict \eqref{4.6}. We first observe that, thanks  to hypothesis $(M_1)$ together with \eqref{4.1} and  \eqref{ineq-I}  we have
	\begin{equation}\label{4.7}
	\begin{aligned}
	C\geq& I(u_j)-\frac{1}{q\theta}\langle I' (u_j),u_j\rangle\\
	=&\mathscr M[\phi_{\mathcal H}(\nabla u_j)]-\frac{1}{q\theta}M[\phi_{\mathcal H}(\nabla u_j)]\varrho_{\mathcal H}(\nabla u_j)
+\left(\frac{1}{p}-\frac{1}{q\theta}\right)\|u_j\|_p^p+\left(\frac{1}{q}-\frac{1}{q\theta}\right)\|u_j\|_{q,a}^q\\
	&-\int_{\Omega}\left[F(x,u_j)-\frac{1}{q\theta}f(x,u_j)u_j\right]dx
-\int_{\partial\Omega}\left[G(x,u_j)-\frac{1}{q\theta}g(x,u_j)u_j\right]d\sigma\\
\geq&\int_\Omega\mathcal F(x,u_j)\;dx+\int_{\partial\Omega}\mathcal G(x,u_j)\;d\sigma.
	\end{aligned}
	\end{equation}
For all $a\geq0$ and $b>a$, we now set $\Omega_j(a,b):=\left\{x\in\Omega:\,a\leq|u_j(x)|<b\right\}$. Thanks to $(h_1)$ and \eqref{4.3} it follows that
	\begin{equation}\label{4.8}
	\begin{aligned}
	\int_{\Omega_j(0,t_0)}\frac{F(x,u_j)}{\|u_j\|_{1,\mathcal H}^p}\;dx&\leq 
c_1\int_{\Omega_j(0,t_0)}\left(\frac{|u_j|}{\|u_j\|_{1,\mathcal H}^p}+\frac{|u_j|^{r_1}}{{r_1}\|u_j\|_{1,\mathcal H}^p}\right)dx\\
	&\leq c_1\left(\frac{\|u_j\|_1}{\|u_j\|_{1,\mathcal H}^p}
+\frac{1}{r_1}\int_{\Omega_j(0,t_0)}|u_j|^{r_1-p}|v_j|^p\;dx\right)\\
	&\leq c_1\left(\frac{C}{\|u_j\|_{1,\mathcal H}^{p-1}}+\frac{t_0^{r_1-p}}{r_1}\|v_j\|_p^p\right)\to0\quad\mbox{as }j\to\infty,
	\end{aligned}
	\end{equation}
being $v=0$.
On the other hand, by $(h_4)$, \eqref{4.7}, and H\"older's inequality, we have
	\begin{equation}\label{4.8bis}
	\begin{aligned}
	\int_{\Omega_j(t_0,\infty)}\frac{F(x,u_j)}{\|u_j\|_{1,\mathcal H}^p}\;dx&=
\int_{\Omega_j(t_0,\infty)}\frac{F(x,u_j)}{|u_j|^p}|v_j|^p \;dx\\
	&\leq\left[\int_{\Omega_j(t_0,\infty)}\left(\frac{F(x,u_j)}{|u_j|^p}\right)^{s_1}dx\right]^{1/s_1}
\left(\int_{\Omega_j(t_0,\infty)}|v_j|^{ps_1'} \;dx\right)^{1/s_1'}\\
	&\leq d_1^{1/s_1}\left(\int_{\Omega_j(t_0,\infty)}\mathcal F(x,u_j) \;dx\right)^{1/s_1}
\|v_j\|_{ps_1'}^p\\
	&\leq d_1^{1/s_1}C^{1/s_1}
\|v_j\|_{ps_1'}^p\to0\quad\mbox{as }j\to\infty,
	\end{aligned}
	\end{equation}
taking once again \eqref{4.3} into account, with $v=0$, and since $ps_1'<p^*$ thanks to $(h_4)$. Combining \eqref{4.8} and \eqref{4.8bis}, we get
	\begin{equation*}\label{finale1}
	\lim_{j\to\infty}\int_\Omega\frac{F(x,u_j)}{\|u_j\|_{1,\mathcal H}^p}dx=0.
	\end{equation*}
Reasoning in a similar way and exploiting the fact that   $ps_2'<p_*$ thanks to hypothesis $(h_4)$, we have
	\begin{equation*}\label{finale2}
	\lim_{j\to\infty}\int_{\partial\Omega}\frac{G(x,u_j)}{\|u_j\|_{1,\mathcal H}^p}d\sigma=0.
	\end{equation*}
Therefore, \eqref{stima-FG} follows, giving the desired contradiction.
This allows us to conclude that $(u_j)_j$ is bounded in $W^{1,\mathcal H}(\Omega)$.
\end{proof}

\begin{lemma}\label{lemma4.2}
The functional $I$ satisfies the $(C)$ condition.
\end{lemma}

\begin{proof}

Let $(u_j)_j\subset W^{1,\mathcal H}(\Omega)$ be a sequence satisfying \eqref{cerami} with $E=I$.
Thanks to Lemma \ref{lemma4.1} we have that $(u_j)_j$ is bounded in $W^{1,\mathcal H}(\Omega)$.
Therefore, taking into account Propositions \ref{P2.1}-\ref{P2.3} and the reflexivity of $W^{1,\mathcal H}(\Omega)$, there exist a subsequence, still denoted by $(u_j)_j$, and $u\in W^{1,\mathcal H}(\Omega)$ such that
	\begin{equation}\label{convergenze}
	\begin{gathered}
	u_j\to u\mbox{ in }L^{\mathcal H}(\Omega),\qquad \nabla u_j\rightharpoonup\nabla u\mbox{ in }\left[L^{\mathcal H}(\Omega)\right]^N, \qquad\phi_{\mathcal H}(\nabla u_j)\rightarrow\ell,\\
	u_j\rightharpoonup u\mbox{ in }\soh,\qquad u_j\to u\mbox{ in }L_a^q(\Omega)\cap L^{\nu_1}(\Omega)\cap L^{\nu_2}(\partial\Omega),
	\end{gathered}
	\end{equation}
as $j\to\infty$, with $\nu_1\in[1,p^*)$ and $\nu_2\in[1,p_*)$. 

We aim to show that such $(u_j)_j$ is strongly convergent in $W^{1,\mathcal H}(\Omega)$. Let us distinguish between two possible situations. We first assume that $\ell=0$. Therefore, since $\phi_{\mathcal H}(v)\geq\varrho_{\mathcal H}(v)/q\geq0$ for all $v\in W^{1,\mathcal H}(\Omega)$, thanks to Proposition \ref{P2.1}-(v) we have $\nabla u_j\to\overline{0}$ in $\left[L^{\mathcal H}(\Omega)\right]^N$. Thus we can conclude that $u_j\to u$ in $W^{1,\mathcal H}(\Omega)$ as $j\to\infty$, with $u$ constant a.e. in $\Omega$.

On the other hand, let us suppose $\ell>0$. Thanks to \eqref{convergenze} and Proposition \ref{P2.4} it suffices to show that
\begin{equation*}
    \limsup_{j\to \infty}\langle L(u_j)-L(u),u_j-u\rangle\leq 0,
\end{equation*}
where $L$ is the functional defined in \eqref{essepiu}.
Taking into account hypothesis $(h_1)$, the boundedness of $(u_j)_j$,  the convergences in \eqref{convergenze}, and applying H\"older's inequality we obtain that, as $j\to\infty$,
	\begin{equation}\label{4.2.1}
	\begin{aligned}
	\left|\int_\Omega f(x,u_j)(u_j-u) \;dx\right|
	&\leq c_1\int_\Omega\left(1+|u_j|^{r_1-1}\right)|u_j-u|\;dx\\
	&\leq c_1\left(\|u_j-u\|_1+\Vert u_j\Vert_{r_1}^{r_1-1}\|u_j-u\|_{r_1}\right)\to 0,
	\end{aligned}
	\end{equation}
as well as
	\begin{equation}\label{4.2.2}
	\begin{aligned}
	\left|\int_{\partial\Omega} g(x,u_j)(u_j-u)\;d\sigma\right|
	&\leq c_2\int_{\partial\Omega}\left(1+|u_j|^{r_2-1}\right)|u_j-u|\;d\sigma\\
	&\leq c_2\left(\|u_j-u\|_{1,\partial\Omega}+\Vert u_j\Vert_{r_2,\partial\Om}^{r_2-1}\|u_j-u\|_{r_2,\partial\Omega}\right)\to 0,
	\end{aligned}
	\end{equation}
similarly, 
	\begin{equation}\label{holder1}
   \left| \int_\Omega |u_j|^{p-2}u_j(u_j-u)\;dx\right|\leq\|u_j\|_p^{p-1}\|u-u_j\|_p\to 0,
	\end{equation}
and finally
	\begin{equation}\label{holder2}
	   \left| \int_\Omega a(x)|u_j|^{q-2}u_j(u_j-u)\;dx\right|\leq\|u_j\|_{q,a}^{q-1}\|u-u_j\|_{q,a}\to 0.
	\end{equation}
 Thus, by means of \eqref{cerami} and \eqref{convergenze}--\eqref{holder2}  we get
	\begin{equation}\label{3.12}
	\begin{aligned}
	o(1)=\langle I'(u_j),u_j-u\rangle=&M\left[\phi_{\mathcal H}(\nabla u_j)\right]\langle L(u_j),u_j-u\rangle\\
	&+\int_\Omega |u_j|^{p-2}u_j(u_j-u)\;dx+\int_\Omega a(x) |u_j|^{q-2}u_j(u_j-u)\;dx\\
	&-\int_\Omega f(x,u_j)(u_j-u)\;dx-\int_{\partial\Omega} g(x,u_j)(u_j-u)\;d\sigma\\
	=&M(\ell)\langle L(u_j),u_j-u\rangle+o(1)\qquad\text{ as }j\to \infty.
	\end{aligned}
	\end{equation}
Moreover, thanks to H\"older's inequality and Proposition \ref{P2.1}, given $\phi\in [L^\cH(\Om)]^N$ such that $\|\phi\|_\cH=1$ it holds that
\begin{equation*}
    \begin{split}
        \left|\int_\Omega\left(|\nabla u|^{p-2}\nabla u+a(x)|\nabla u|^{q-2}\nabla u\right)\cdot \phi\;dx\right|&\leq\into|\nabla u|^{p-1}|\phi|\;dx+\into a(x)^{\frac{q-1}{q}}|\nabla u|^{q-1}a(x)^{\frac{1}{q}}|\phi|\;dx\\
        &\leq\|\nabla u\|_p^{p-1}\|\phi\|_p+\|\nabla u\|_{q,a}^{q-1}\|\phi\|_{q,a}\\
        &\leq\max\{\|\nabla u\|_p^{p-1}, \|\nabla u\|_{q,a}^{q-1}\}\varrho_\cH(\phi)\\
        &=\max\{\|\nabla u\|_p^{p-1},\|\nabla u\|_{q,a}^{q-1}\}.
    \end{split}
\end{equation*}
This implies that the functional
	\[
	P\colon \phi\in\left[L^{\mathcal H}(\Omega)\right]^N\mapsto\int_\Omega\left(|\nabla u|^{p-2}\nabla u+a(x)|\nabla u|^{q-2}\nabla u\right)\cdot \phi\,dx
	\]
is linear and bounded.
Therefore, by \eqref{convergenze} we have
	\begin{equation}\label{3.13}
	\langle L(u),u_j-u\rangle
	=\int_\Omega\left(|\nabla u|^{p-2}\nabla u+a(x)|\nabla u|^{q-2}\nabla u\right) \cdot \nabla (u_j-  u) \; dx\to 0\mbox{ as }j\to\infty.
	\end{equation}
Combining \eqref{3.12}-\eqref{3.13}, using Proposition \ref{P2.4} and taking into account that $M(\ell)>0$ thanks to $(M_2)$,  we conclude that $u_j\to u$ in $W^{1,\mathcal H}(\Omega)$ as $j\to\infty$. This completes the proof.
\end{proof}

We now point out that, since $W^{1, \Hi}(\Om)$ is a reflexive and separable Banach space, there exist two sequences $(e_j)_j \subset W^{1, \Hi}(\Om)$ and $(e^*_j)_j \subset \l(W^{1, \Hi}(\Om)\r)^*$ that satisfy \eqref{fountain1}--\eqref{3.15}.

Then we can state the following lemma, which is strongly inspired by \cite[Lemma 7.1]{LD}. First of all, for all $j \in \NN$ we set
	\begin{equation}
	\label{def-xi}
	\beta_j:=\sup_{u\in Z_j,\,\normasoh{u}=1} \|u\|_{r_1} \quad \text{as well as} \quad \xi_j:=\sup_{u\in Z_j,\,\normasoh{u}=1} \|u\|_{r_2,\partial\Om},
	\end{equation}
where $Z_j$ is defined in \eqref{3.15} and  $r_1$, $r_2$ are chosen as in $(h_1)$.

\begin{lemma}\label{asintotica}
It holds that 
	\[
	\lim_{j\to\infty}\beta_j=\lim_{j\to\infty}\xi_j=0.
	\]
\end{lemma}

\begin{proof}

The claim for $(\beta_j)_j$ is the content of \cite[Lemma 7.1]{LD}, then we are left to show that $\displaystyle \lim_{j\to\infty}\xi_j=0$. From  \eqref{def-xi}, for all $j\in\NN$ we choose $u_j\in Z_j$ such that $\normasoh{u_j}=1$ and 	
	\begin{equation}
	\label{xi}
	\xi_j\leq\|u_j\|_{r_2,\partial\Om}+\frac{1}{j}.
	\end{equation}
 Since $(u_j)_j$ is bounded in the reflexive space $\soh$, thanks to Proposition \ref{P2.3}-(iii)  there exists $u\in\soh$ such that 
	\begin{equation}\label{dueconv}
u_j\rightharpoonup u\mbox{ in }\soh,\qquad u_j\to u\mbox{ in }L^{r_2}(\partial\Om).
	\end{equation}
Moreover,  fix $k\in\NN$. From \eqref{fountain2} it follows that   $\langle e^*_k,u_j\rangle=0$, for all $j\in\NN$ big enough. Then we have 
 	\[
	\langle e^*_k,u\rangle=\lim_{j\to\infty}\langle e^*_k,u_j\rangle=0 \quad\mbox{for all }k\in\NN.
	\]
This implies that $u=0$ in $\soh$. Then, by \eqref{dueconv} it holds that $u_j\to0$ as $j\to\infty$ in $L^{r_2}(\partial\Om)$. Therefore, from \eqref{xi} we have the conclusion. 
\end{proof}




We are now ready to prove Theorem \ref{thm:infinitely}.

\begin{proof}[Proof of Theorem \ref{thm:infinitely}]

Thanks to hypothesis $(h_5)$ and  Lemma \ref{lemma4.2} we have that $I$ is an even functional and satisfies  the $(C)$ condition.
Thus we are left to verify that conditions (i) and (ii) of Theorem \ref{thm:fountain} hold true.
We start with condition (i).
By $(h_1)$ we have that
	\begin{equation}\label{3.1n}
	F(x,t)\leq c_1\left(|t|+\frac{|t|^{r_1}}{r_1}\right) \quad\mbox{for a.e. } x\in\Omega 
	\end{equation}
as well as
	\begin{equation}\label{3.1nbis}
	G(x,t)\leq c_2\left(|t|+\frac{|t|^{r_2}}{r_2}\right) \quad \mbox{for a.e. }x\in\partial\Omega, 
	\end{equation}
for all $t \in \RR$. Thus, for all $u\in Z_j$ with $\normasoh{u}>1$, thanks to \eqref{stimaA1}, hypothesis $(M_1)$, \eqref{def-xi}, \eqref{3.1n}, \eqref{3.1nbis}, and the H\"older's inequality,  we have
	\begin{equation}\label{3.19}
	\begin{aligned}
		I(u)&\geq\frac{1}{q\theta}\l(M[\phi_{\mathcal H}(\nabla u)]\phi_{\mathcal H}(\nabla u)+\varrho_\mathcal{H}(u)\r)-c_1\left(\|u\|_{1}+\frac{\|u\|_{r_1}^{r_1}}{r_1}\right)+c_2\left(\Vert u\Vert_{1,\partial\Om}+\frac{\Vert u\Vert_{r_2,\partial\Om}^{r_2}}{r_2}\right)\\
	&\geq\frac{{A_1}}{q\theta}\normasoh{u}^p
-c_1\left(|\Omega|^{(r_1-1)/r_1}\|u\|_{r_1}+\frac{\|u\|_{r_1}^{r_1}}{r_1}\right)+c_2\left(\sigma(\Om)^{(r_2-1)/r_2}\Vert u\Vert_{r_2,\partial\Om}+\frac{\Vert u\Vert_{r_2,\partial\Om}^{r_2}}{r_2}\right)\\
	&\geq {C}\left(\normasoh{u}^p-\beta_j\normasoh{u}-\beta_j^{r_1}\normasoh{u}^{r_1}-\xi_j\normasoh{u}-\xi_j^{r_2}\normasoh{u}^{r_2}\right).
	\end{aligned}
	\end{equation}
	for a suitable $C>0$.
	
Set now $r:=\max\{r_1,r_2\}>p$ and $\eta_j:=\max\{\beta_j,\xi_j\}$, with $\beta_j$ and $\xi_j$ given in \eqref{def-xi}. By Lemma \ref{asintotica}, we have that $\eta_j<1$ if $j\in\mathbb N$ is sufficiently large. Hence, by \eqref{3.19}, for all $u\in Z_j$ with $\normasoh{u}>1$ and $j\in\mathbb N$ big enough, we get
	\begin{equation}\label{3.19bis}
	\begin{aligned}
	I(u)\geq{C}(1-4\eta_j\normasoh{u}^{r-p})\normasoh{u}^p.
	\end{aligned}
	\end{equation}
Let us choose
	\[
	\gamma_j:=\left(\frac{1}{8\eta_j}\right)^\frac{1}{r-p}.
	\]
Then $\gamma_j\to\infty$ as $j\to\infty$, since $\eta_j\to 0$ as $j\to\infty$ thanks to Lemma \ref{asintotica} and the fact that $r>p$.
Inequality \eqref{3.19bis} yields that, for all $u\in Z_j$ with $\normasoh{u}=\gamma_j$, 
	\[
	I(u)\geq\frac{\widetilde{C}}{2}\gamma_j^p\to\infty \quad \mbox{as }j\to\infty,
	\]
which gives the validity of condition (i).

In order to prove condition (ii) we argue by contradiction. Therefore we assume that there exists $j>0$ such that for each $n\in\mathbb{N}\setminus\{0\}$ we can find a function $u_n\in Y_j$ with $\|u_n\|_{1,\mathcal{H}}>n$ and $I(u_n)>0$. Let us define, for each $n$, 
$v_n:=u_n/\|u_n\|_{1,\mathcal{H}}$. Of course, $\|v_n\|_{1,\mathcal{H}}=1$. Since $(Y_j,\|\cdot\|_{1,\mathcal{H}})$ is a finite dimensional Banach space, it follows that there exists $v\in Y_j$ such that, up to a subsequence, 
$$\|v_n-v\|_{1,\mathcal{H}}\to 0$$
as $n \to \infty$. Since also $\|v\|_{1,\mathcal{H}}=1$,  setting $A:=\{x\in\Om\,:\,v(x)\neq 0\}$, it follows that $|A|\neq 0$. 
Therefore, arguing as in the proof of Lemma \ref{lemma4.1}, thanks to $(h_2)$ we conclude that 
\begin{equation}\label{contradictionfine}
    \lim_{n\to\infty}\int_A\frac{F(x,u_n(x))}{\|u_n\|_{1,\mathcal{H}}^{q\theta}}\;dx=+\infty.
\end{equation}
On the other hand, as $I(u_n)>0$, and thanks to the nonnegativity of $G$ and \eqref{stimaB1}, we have that
\begin{equation}\label{stimainfinito2}
    \begin{split}
        \int_AF(x,u_n(x))dx&\leq\int_\Om F(x,u_n(x))dx\leq\int_\Om F(x,u_n(x))dx+I(u_n)\\
        &=\mathscr M\l[\phi_{\mathcal H}(\nabla u_n)\r]+ \frac{1}{p} \|u_n\|_p^p+ \frac{1}{q} \|u_n\|_{q,a}^q- \intor G(x, u_n) \;d\sigma\\
        &\leq \mathscr M\l[\phi_{\mathcal H}(\nabla u_n)\r]+ \frac{1}{p} \|u_n\|_p^p+ \frac{1}{q} \|u_n\|_{q,a}^q\\
        &\leq B_1(1+\|u_n\|_{1, \cH}^q+\|u_n\|_{1, \cH}^{q\theta}).
    \end{split}
\end{equation}
Dividing both sides by $\|u_n\|^{q\theta}_{1,\mathcal{H}}$ and passing to the limit, a contradiction with \eqref{contradictionfine} follows.




Thus, we can apply Theorem \ref{thm:fountain}  to  obtain a sequence of critical points of $I$ with unbounded energy. The proof is thus complete. 
\end{proof}

\medskip

We now aim to show  existence of infinitely many solutions to \eqref{P2}, 
where $h_1$ and $h_2$ are chosen as in \eqref{varieh}, that is, 
	\begin{equation}\label{problem3}
	\left\{
	\begin{aligned}
	& -M\left(\|\nabla u\|_p^p\right)\Delta_p u\!-\!M\left(\|\nabla u\|_{q,a}^q\right)\mbox{div}\left(a(x)|\nabla u|^{q-2}\nabla u\right)
	= f(x, u)\\
	&\qquad\qquad\qquad\qquad\qquad\qquad\qquad\qquad\qquad\qquad\qquad\qquad\quad- |u|^{p-2} u- a(x) |u|^{q-2} u \quad && \mbox{in } \Omega,\\
	& \left[M\left(\|\nabla u\|_p^p\right)|\nabla u|^{p-2}\nabla u+M\left(\|\nabla u\|_{q,a}^q\right)a(x)|\nabla u|^{q-2}\nabla u\right]\cdot \nu = g(x,u) && \mbox{on } \partial\Omega.
	\end{aligned}
	\right.
	\end{equation}
We say that a function $u\in W^{1,\mathcal H}(\Omega)$  is a weak solution of \eqref{problem3} if
\begin{equation*}
    \begin{split}
   M(\|\nabla u\|_p^p)&\int_\Omega |\nabla u|^{p-2}\nabla u\cdot\nabla\varphi \;dx+
M(\|\nabla u\|_{q,a}^q)\int_\Omega a(x)|\nabla u|^{q-2}\nabla u\cdot\nabla\varphi \;dx
+ \into \l(|u|^{p-2}u+ a(x) |u|^{q-2}u \r)\varphi \;dx\\
	&= \into f(x, u) \varphi \;dx+ \intor  g(x,u)  \varphi \;d\sigma,
\end{split}
\end{equation*}
is satisfied for all $\varphi\in W^{1,\mathcal H}(\Omega)$. In this case, \eqref{problem3} is the Euler-Lagrange equation associated with the energy functional $\mathcal I:W^{1,\mathcal H}(\Omega)\to\mathbb R$ given by
$$\mathcal I(u):=\frac{1}{p}\mathscr M(\|\nabla u\|_p^p)+\frac{1}{q}\mathscr M(\|\nabla u\|_{q,a}^q)
+\frac{1}{p}\|u\|_p^p+\frac{1}{q}\|u\|_{q,a}^q-\int_\Omega F(x,u)\;dx-\int_{\partial\Omega} G(x,u)\;d\sigma.
$$
Our multiplicity result for \eqref{problem3}  reads as follows. 

\begin{theorem}\label{thm:infinitely2}
Let \eqref{condizione1}, $(M_1)$-$(M_2)$, and $(h_1)$-$(h_5)$ hold true.
Then, problem \eqref{problem3} has infinitely many weak solutions $(u_j)_j$ with unbounded energy.
\end{theorem}

Here, we point out that assumption $(h_3)$ could allow us to prove the boundedness of a Palais-Smale sequence for $\mathcal I$, but just in $W^{1,p}(\Omega)$. This fact is not enough considering $\mathcal I$ set in $W^{1,\mathcal H}(\Omega)$, with $W^{1,\mathcal H}(\Omega)\hookrightarrow W^{1,p}(\Omega)$ by Proposition \ref{P2.3}-(i).
For this, we exploit the same ideas employed in the proof of Theorem \ref{thm:infinitely} in order to prove Theorem \ref{thm:infinitely2}. That is, we start by showing that Cerami sequences of $\mathcal I$ are bounded. Then we use this property to show that $\mathcal I$ satisfies the ($C$) condition. Finally we apply  Theorem \ref{thm:fountain} to $\mathcal I$.

\begin{lemma}\label{lemma5.1}
Any Cerami sequence of $\mathcal I$ is bounded in $W^{1,\mathcal H}(\Omega)$.
\end{lemma}

\begin{proof}
The proof works exactly as in Lemma \ref{4.1}, with \eqref{differenza1} and \eqref{4.7} following  from \eqref{stimaA2} and \eqref{stimaB2}, respectively.
\end{proof}

\begin{lemma}\label{lemma5.2}
The functional $\mathcal I$ satisfies the $(C)$ condition.
\end{lemma}

\begin{proof}
Let $(u_j)_j\subset W^{1,\mathcal H}(\Omega)$ be a sequence satisfying \eqref{cerami} with $E=\mathcal I$. 
Thanks to Lemma \ref{lemma5.1}, $(u_j)_j$ is bounded in $W^{1,\mathcal H}(\Omega)$.
Hence, by Propositions \ref{P2.1}-\ref{P2.3}  and the reflexivity of $W^{1,\mathcal H}(\Omega)$, there exist a subsequence, still denoted by $(u_j)_j$, and $u\in W^{1,\mathcal H}(\Omega)$ such that
	\begin{equation}\label{convergenze2}
	\begin{gathered}
	u_j\to u\mbox{ in }L^{\mathcal H}(\Omega),
	\qquad\nabla u_j\rightharpoonup\nabla u\mbox{ in }\left[L^{\mathcal H}(\Omega)\right]^N,
	\qquad\|\nabla u_j\|_p\rightarrow\ell_p,\\
	\|\nabla u_j\|_{q,a}\rightarrow\ell_q,\qquad u_j\rightharpoonup u\mbox{ in }\soh,\qquad u_j\to u\mbox{ in }L^{\nu_1}(\Omega)\cap L^{\nu_2}(\partial\Omega),
	\end{gathered}
	\end{equation}
as $j\to\infty$, with $\nu_1\in[1,p^*)$ and $\nu_2\in[1,p_*)$.
We apply \eqref{cerami} together with \eqref{4.2.1}--\eqref{holder2} and \eqref{convergenze2}, to have
	\begin{align}\label{n3.12}
	o(1)=\langle \mathcal I'(u_j),u_j-u\rangle=&M(\|\nabla u_j\|_p^p)\int_\Omega|\nabla u_j|	^{p-2}\nabla u_j\cdot(\nabla u_j-\nabla u)\;dx\nonumber\\
	&+M(\|\nabla u_j\|_{q,a}^q)\int_\Omega a(x)|\nabla u_j|^{q-2}\nabla u_j\cdot(\nabla u_j-\nabla u)\;dx\nonumber\\
	&+\int_\Omega |u_j|^{p-2}u_j(u_j-u)\;dx+\int_\Omega a(x) |u_j|^{q-2}u_j(u_j-u)\;dx\\
	&-\int_\Omega f(x,u_j)(u_j-u)\;dx-\int_{\partial\Omega} g(x,u_j)(u_j-u)\;d\sigma\nonumber\\
	=&M(\ell_p^p)\int_\Omega|\nabla u_j|^{p-2}\nabla u_j\cdot(\nabla u_j-\nabla u)\;dx\nonumber\\
	&+M(\ell_q^q)\int_\Omega a(x)|\nabla u_j|^{q-2}\nabla u_j\cdot(\nabla u_j-\nabla u)\;dx+o(1)\nonumber\quad\text{as $j\to\infty$}.
	\end{align}
We need now to distinguish two situations, that depend on  the behavior of $M$ at zero.

\begin{case}\label{step1}
{\em Let $M$ verify $M(0)=0$.}
\end{case} 
\noindent
Since $\ell_p\geq0$ and $\ell_q\geq0$ in \eqref{convergenze2}, we further distinguish  four subcases.

\begin{subcase}
{\em Let $\ell_p=0$ and $\ell_q=0$.}
\end{subcase}
\noindent
By \eqref{convergenze2} we have $\|\nabla u_j\|_p\to0$ and $\|\nabla u_j\|_{q,a}\to0$ as $j\to\infty$, which thanks to \eqref{rhoh} and Proposition \ref{P2.1} implies that $\nabla u_j\to\overline{0}$ in $[L^{\mathcal H}(\Omega)]^N$. Hence, $u_j\to u$ in $W^{1,\mathcal H}(\Omega)$ as $j\to\infty$, with $u$ constant a.e. in $\Omega$.

\begin{subcase}\label{case2}
{\em Let $\ell_p=0$ and $\ell_q>0$.}
\end{subcase}
\noindent
From \eqref{n3.12} and $(M_2)$ we have
	\begin{equation}\label{1944}
	\lim_{j\to\infty}\int_\Omega a(x)|\nabla u_j|^{q-2}\nabla u_j\cdot(\nabla u_j-\nabla u)\;dx=0.
	\end{equation}
Moreover, since $(\nabla u_j)_j$ is bounded in $(L^p(\Om))^N$ thanks to \eqref{convergenze2}, it follows that
	\begin{equation*}
	\begin{split}\label{partelp}
    \limsup_{j\to\infty}\left|\int_\Omega |\nabla u_j|^{p-2}\nabla u_j\cdot(\nabla u_j-\nabla u)\;dx\right|&\leq\limsup_{j\to\infty}\|\nabla u_j\|_p^{p-1}\|\nabla u_j-\nabla u\|_p\\
    &\leq\ell_p^{p-1} \limsup_{j\to\infty}\|\nabla u_j-\nabla u\|_p=0.
   \end{split}
	\end{equation*}
Then we have
	\[
	\lim_{j\to\infty}\langle L(u_j),u_j-u\rangle=0.
	\]
This fact, together with \eqref{3.13} and Proposition \ref{P2.4}, allows to conclude that $u_j\to u$ in $\soh$.

\begin{subcase}\label{case3}
{\em Let $\ell_p>0$ and $\ell_q=0$.}
\end{subcase}
\noindent
The proof works exactly as in Subcase \ref{case2}, hence we omit the details.

\begin{subcase}\label{case4}
{\em Let $\ell_p>0$ and $\ell_q>0$.}
\end{subcase}
\noindent
It follows that
	\begin{equation}\label{n3.122}
	\lim_{j\to\infty}\int_\Omega |\nabla u|^{p-2}\nabla u\cdot \nabla(u_j-  u) \;dx=0,
	\end{equation}
as well as 
	\begin{equation}\label{mancante}
    \lim_{j\to\infty}\into a(x)|\nabla u|^{q-2}\nabla u\cdot \nabla(u_j- u)\;dx=0.
	\end{equation}
Therefore \eqref{n3.12}, \eqref{n3.122}, and \eqref{mancante} yield    
	\begin{equation}\label{n3.13}
	\begin{aligned}
	M&(\ell_p^p)\int_\Omega\left(|\nabla u_j|^{p-2}\nabla u_j-|\nabla u|^{p-2}\nabla u\right)\cdot(\nabla u_j-\nabla u)\;dx\\
	&+M(\ell_q^q)\int_\Omega a(x)\left(|\nabla u_j|^{q-2}\nabla u_j-|\nabla u|^{q-2}\nabla u\right)\cdot(\nabla u_j-\nabla u)\;dx=o(1)\quad\text{as $j\to\infty$.}
	\end{aligned}
	\end{equation}
Exploiting the convexity of the two Laplacian operators of $p$ and $q$ types and the fact that $a(x)\geq0$ a.e. in $\Omega$ thanks to \eqref{condizione1}, we get
	\[
	\begin{aligned}
	&\left(|\nabla u_j|^{p-2}\nabla u_j-|\nabla u|^{p-2}\nabla u\right)\cdot(\nabla u_j-\nabla u)
\geq0\mbox{ a.e. in }\Omega,\\
	&a(x)\left(|\nabla u_j|^{q-2}\nabla u_j-|\nabla u|^{q-2}\nabla u\right)\cdot(\nabla u_j-\nabla u) \geq0\mbox{ a.e. in }\Omega
	\end{aligned}
	\]
Therefore, from  \eqref{n3.13} we have
	\[
	\min\left\{M(\ell_p^p),M(\ell_q^q)\right\}\limsup_{j\to\infty}\langle L(u_j)-L(u),u_j-u\rangle\leq0,
	\]
being $M(\ell_p^p)$, $M(\ell_q^q)>0$ thanks to $(M_2)$. Hence we can use Proposition \ref{P2.4} and \eqref{convergenze2} to conclude that $u_j\to u$ in $W^{1,\mathcal H}(\Omega)$ as $j\to\infty$. This completes the proof of Case \ref{step1}.

\begin{case}
{\em Let $M$ verify $M(0)>0$.}
\end{case}
\noindent 
Since by $(M_2)$ we have that $M(\ell_p^p)$, $M(\ell_q^q)>0$ for $\ell_p$, $\ell_q\geq0$, we can argue exactly as in  Subcase \ref{case4} to get the conclusion. 
\end{proof}

\begin{proof}[Proof of Theorem \ref{thm:infinitely2}]

The proof works exactly as for Theorem \ref{thm:infinitely}, with  \eqref{3.19bis} and \eqref{stimainfinito2} following from \eqref{stimaA2} and \eqref{stimaB2},  respectively.
\end{proof}

\section*{Acknowledgments}

The authors are members of {\em Gruppo Nazionale per l'Analisi Ma\-te\-ma\-ti\-ca, la Probabilit\`a e le loro Applicazioni} (GNAMPA) 
of the {\em Istituto Nazionale di Alta Matematica} (INdAM).

A. Fiscella is partially supported by INdAM-GNAMPA project titled {\em Equazioni alle derivate parziali: problemi e modelli} and by the FAPESP Thematic Project titled {\em Systems and partial differential equations} (2019/02512-5).

A. Pinamonti is partially supported by the INdAM-GNAMPA project {\em Convergenze variazionali per funzionali e operatori dipendenti da campi vettoriali}.

\end{document}